\documentclass{article}
\usepackage[margin=2cm]{geometry}
\usepackage[T1]{fontenc}
\usepackage[utf8]{inputenc}
\usepackage{amsmath,amssymb,amsthm}
\usepackage[shortlabels]{enumitem}
\usepackage{color,url}
\usepackage{hyperref}
\usepackage[capitalise]{cleveref}
\usepackage{xspace}
\usepackage{todonotes}
\usepackage{thm-restate}
\usepackage{comment}
\usepackage{xcolor}
\usepackage{colortbl}
\usepackage{authblk}

\theoremstyle{plain}
\newtheorem{theorem}{Theorem}[section]
\newtheorem{lemma}[theorem]{Lemma}
\newtheorem{proposition}[theorem]{Proposition}
\newtheorem{observation}[theorem]{Observation}

\newtheorem{definition}[theorem]{Definition}
\newtheorem{claim}{Claim}[theorem]
\newenvironment{claimproof}{\noindent\textit{Proof of Claim \theclaim:}}{\hfill$\lrcorner$} 

\newcommand{\Ram}{\mathsf{Ram}}
\renewcommand{\epsilon}{\varepsilon}
\newcommand{\N}{\mathbb{N}}

\newcommand{\Oh}{\mathcal{O}}
\newcommand{\Ind}{\mathsf{Ind}}

\newcommand{\cC}{\mathcal{C}}

\newcommand{\col}[1]{{#1}-\textsc{Coloring}\xspace}
\newcommand{\lcol}[1]{\textsc{List} {#1}-\textsc{Coloring}\xspace}
\newcommand{\listcol}{\textsc{List} \textsc{Coloring}\xspace}

\newcommand{\sat}[1]{{#1}-\textsc{Sat}\xspace}

\usepackage{todonotes}

\newcommand{\wave}{
\begin{tikzpicture}[every node/.style={fill=black, circle, inner sep=0.5pt, outer sep = 0}, scale=0.2]
    \node (a1) at (0,0) {};
    \node (b1) at (1,0) {};
    \node (a2) at (2,0) {};    
    \node (b2) at (3,0) {};    
    \path[draw] (a1) to[bend left = 60] (a2);
    \path[draw] (b1) to[bend left = 60] (b2);
\end{tikzpicture}
}

\newcommand{\rainbow}{
\begin{tikzpicture}[every node/.style={fill=black, circle, inner sep=0.5pt, outer sep = 0}, scale=0.2]
    \node (a1) at (0,0) {};
    \node (b1) at (1,0) {};
    \node (b2) at (2,0) {};    
    \node (a2) at (3,0) {};    
    \path[draw] (a1) to[bend left = 60] (a2);
    \path[draw] (b1) to[bend left = 45] (b2);
\end{tikzpicture}
}

\newcommand{\jj}{
\begin{tikzpicture}[every node/.style={fill=black, circle, inner sep=0.5pt, outer sep = 0}, scale=0.2]
    \node (x) at (0,0) {};
    \node (y) at (1,0) {};
    \node (z) at (2,0) {};    
    \path[draw] (x) to[bend left = 45] (y);
    \path[draw] (x) to[bend left = 60] (z);
\end{tikzpicture}
}

\newcommand{\jjrev}{
\begin{tikzpicture}[every node/.style={fill=black, circle, inner sep=0.5pt, outer sep = 0}, scale=0.2]
    \node (x) at (0,0) {};
    \node (y) at (1,0) {};
    \node (z) at (2,0) {};    
    \path[draw] (y) to[bend left = 45] (z);
    \path[draw] (x) to[bend left = 60] (z);
\end{tikzpicture}
}

\newcommand{\lj}{_{\ell}\jj}

\newcommand{\jp}{\jj_{1}}

\newcommand{\pthree}{
\begin{tikzpicture}[every node/.style={fill=black, circle, inner sep=0.5pt, outer sep = 0}, scale=0.2]
    \node (x) at (0,0) {};
    \node (y) at (1,0) {};
    \node (z) at (2,0) {};    
    \path[draw] (x) to[bend left = 45] (y);
    \path[draw] (y) to[bend left = 45] (z);
\end{tikzpicture}
}

\newcommand{\ttriangle}{
\begin{tikzpicture}[every node/.style={fill=black, circle, inner sep=0.5pt, outer sep = 0}, scale=0.2]
    \node (x) at (0,0) {};
    \node (y) at (1,0) {};
    \node (z) at (2,0) {};    
    \path[draw] (x) to[bend left = 45] (y);
    \path[draw] (y) to[bend left = 45] (z);
    \path[draw] (x) to[bend left = 60] (z);
\end{tikzpicture}
}

\newcommand{\jw}{
\begin{tikzpicture}[every node/.style={fill=black, circle, inner sep=0.5pt, outer sep = 0}, scale=0.2]
    \node (x) at (0,0) {};
    \node (y) at (3,0) {};
    \node[draw=none,fill=none] (l1) at (-0.5,-0.2) {\footnotesize{$\ell$}};
     \node[draw=none,fill=none] (l1) at (1.5,-0.2) {\footnotesize{$\ell$}};
      \node[draw=none,fill=none] (l1) at (3.5,-0.2) {\footnotesize{$\ell$}};
    \path[draw] (x) to[bend left = 60] (y);
\end{tikzpicture}
}

\newcommand{\jwp}{
\begin{tikzpicture}[every node/.style={fill=black, circle, inner sep=0.5pt, outer sep = 0}, scale=0.2]
    \node (x) at (0,0) {};
    \node (y) at (3,0) {};
     \node[draw=none,fill=none] (l1) at (1.5,-0.2) {\footnotesize{$\ell$}};
    \path[draw] (x) to[bend left = 60] (y);
\end{tikzpicture}
}

\newcommand{\picJw}{
\begin{tikzpicture}[every node/.style={fill=black, circle, inner sep=0.9pt, outer sep = 0}, scale=0.2]
    \node (a1) at (0,0) {};
    \node (a2) at (6,0) {};    
    \path[draw] (a1) to[bend left = 60] (a2);
    \node[draw=gray,fill=none] (b1) at (-4,0) {};
    \node[draw=gray,fill=none] (b2) at (-3,0) {};
    \node[draw=gray,fill=none] (b3) at (-2,0) {};
    \node[draw=gray,fill=none] (b4) at (2,0) {};
    \node[draw=gray,fill=none] (b5) at (3,0) {};
    \node[draw=gray,fill=none] (b6) at (4,0) {};
    \node[draw=gray,fill=none] (b7) at (8,0) {};
    \node[draw=gray,fill=none] (b8) at (9,0) {};
    \node[draw=gray,fill=none] (b9) at (10,0) {};
    
\end{tikzpicture}
}

\newcommand{\picChordal}{
\begin{tikzpicture}[every node/.style={fill=black, circle, inner sep=0.9pt, outer sep = 0}, scale=0.2]
    \node (a1) at (0,0) {};
    \node (a2) at (3,0) {}; 
    \node (a3) at (6,0) {};
    \path[draw] (a1) to[bend left = 60] (a2);
    \path[draw] (a1) to[bend left = 60] (a3);
\end{tikzpicture}
}

\newcommand{\picChordalx}{
\begin{tikzpicture}[every node/.style={fill=black, circle, inner sep=0.9pt, outer sep = 0}, scale=0.2]
    \node (a1) at (0,0) {};
    \node (a2) at (3,0) {}; 
    \node (a3) at (6,0) {};
    \node (a4) at (-1,0) {};
    \path[draw] (a1) to[bend left = 60] (a2);
    \path[draw] (a1) to[bend left = 60] (a3);
    \node[draw=gray,fill=none] (b1) at (-4,0) {};
    \node[draw=gray,fill=none] (b2) at (-3,0) {};
    \node[draw=gray,fill=none] (b3) at (-2,0) {};
    \node[draw=none,fill=none] (x) at (10,0) {};
\end{tikzpicture}
}

\newcommand{\picChordaly}{
\begin{tikzpicture}[every node/.style={fill=black, circle, inner sep=0.9pt, outer sep = 0}, scale=0.2]
    \node (a1) at (0,0) {};
    \node (a2) at (3,0) {}; 
    \node (a3) at (6,0) {};
    \node (a4) at (7,0) {};
    \path[draw] (a1) to[bend left = 60] (a2);
    \path[draw] (a1) to[bend left = 60] (a3);
    \node[draw=gray,fill=none] (b1) at (-4,0) {};
    \node[draw=gray,fill=none] (b2) at (-3,0) {};
    \node[draw=gray,fill=none] (b3) at (-2,0) {};
    \node[draw=gray,fill=none] (b4) at (8,0) {};
    \node[draw=gray,fill=none] (b5) at (9,0) {};
    \node[draw=gray,fill=none] (b6) at (10,0) {};
\end{tikzpicture}
}

\newcommand{\mseven}{
\begin{tikzpicture}[every node/.style={fill=black, circle, inner sep=0.9pt, outer sep = 0}, scale=0.2]
    \node (a1) at (0,0) {};
    \node (b1) at (1.5,0) {};
    \node (b2) at (4.5,0) {};    
    \node (a2) at (6,0) {};    
    \path[draw] (a1) to[bend left = 60] (a2);
    \path[draw] (b1) to[bend left = 45] (b2);
    \node[draw=gray,fill=none] (c1) at (-4,0) {};
    \node[draw=gray,fill=none] (c2) at (-3,0) {};
    \node[draw=gray,fill=none] (c3) at (-2,0) {};
    \node[draw=gray,fill=none] (c4) at (8,0) {};
    \node[draw=gray,fill=none] (c5) at (9,0) {};
    \node[draw=gray,fill=none] (c6) at (10,0) {};
\end{tikzpicture}
}

\newcommand{\meight}{
\begin{tikzpicture}[every node/.style={fill=black, circle, inner sep=0.9pt, outer sep = 0}, scale=0.2]
    \node (a1) at (0,0) {};
    \node (b1) at (1.5,0) {};
    \node (d) at (3,0) {};
    \node (b2) at (4.5,0) {};    
    \node (a2) at (6,0) {};    
    \path[draw] (a1) to[bend left = 60] (a2);
    \path[draw] (b1) to[bend left = 45] (b2);
    \node[draw=gray,fill=none] (c1) at (-4,0) {};
    \node[draw=gray,fill=none] (c2) at (-3,0) {};
    \node[draw=gray,fill=none] (c3) at (-2,0) {};
    \node[draw=gray,fill=none] (c4) at (8,0) {};
    \node[draw=gray,fill=none] (c5) at (9,0) {};
    \node[draw=gray,fill=none] (c6) at (10,0) {};
\end{tikzpicture}
}

\newcommand{\msix}{
\begin{tikzpicture}[every node/.style={fill=black, circle, inner sep=0.9pt, outer sep = 0}, scale=0.2]
    \node (a1) at (0,0) {};
    \node (b1) at (3,0) {};
    \node (a2) at (6,0) {};    
    \node (b2) at (9,0) {};    
    \path[draw] (a1) to[bend left = 60] (a2);
    \path[draw] (b1) to[bend left = 60] (b2);
    \node[draw=gray,fill=none] (c1) at (-3,0) {};
    \node[draw=gray,fill=none] (c2) at (-2,0) {};
    \node[draw=gray,fill=none] (c3) at (-1,0) {};
    \node[draw=gray,fill=none] (c4) at (1,0) {};
    \node[draw=gray,fill=none] (c5) at (2,0) {};
    \node[draw=gray,fill=none] (c6) at (4,0) {};
    \node[draw=gray,fill=none] (c7) at (5,0) {};
    \node[draw=gray,fill=none] (c8) at (7,0) {};
    \node[draw=gray,fill=none] (c9) at (8,0) {};
    \node[draw=gray,fill=none] (c10) at (10,0) {};
    \node[draw=gray,fill=none] (c11) at (11,0) {};
    \node[draw=gray,fill=none] (c12) at (12,0) {};
\end{tikzpicture}
}

\date{}

\begin{document}

\title{List coloring ordered graphs with forbidden induced subgraphs}

\author[1]{Marta Piecyk}
\author[2,3]{Pawe{\l} Rz\k{a}\.zewski\thanks{Supported by the National Science Centre grant number 2024/54/E/ST6/00094.}}

\affil[1]{CISPA Helmholtz Center for Information Security}
\affil[2]{Warsaw University of Technology}
\affil[3]{University of Warsaw}

\maketitle

\begin{abstract}
In the \textsc{List $k$-Coloring} problem we are given a graph whose every vertex is equipped with a list, which is a subset of $\{1,\ldots,k\}$.
We need to decide if $G$ admits a proper coloring, where every vertex receives a color from its list.

The complexity of the problem in classes defined by forbidding induced subgraphs is a widely studied topic in algorithmic graph theory.
Recently, Hajebi, Li, and Spirkl [SIAM J. Discr. Math. 38 (2024)] initiated the study of  \textsc{List $3$-Coloring} in \emph{ordered graphs}, i.e., graphs with  fixed linear ordering of vertices.
Forbidding ordered induced subgraphs allows us to investigate the boundary of tractability more closely.

We continue this direction of research, focusing mostly on the case of  \textsc{List $4$-Coloring}.
We present several algorithmic and hardness results,
which altogether provide an almost complete dichotomy for classes defined by forbidding one fixed ordered graph: our investigations leave one minimal open case.
\end{abstract}

\section{Introduction}
Coloring is one of the best studied problems in graph theory, both from structural and from algorithmic point of view.
In this paper we are interested in the \emph{list} variant of the problem.
For a fixed integer $k$, an instance of the \lcol{$k$} problem is a pair $(G,L)$, where $G$ is a graph and $L$ is a function mapping each vertex of $G$ to a subset of $\{1,\ldots,k\}$ called \emph{list}.
We ask if $G$ admits a proper coloring where every vertex $v$ receives a color from its list $L(v)$.
Clearly, this problem is \textsf{NP}-hard for all $k \geq 3$ as it generalizes the \col{$k$} problem.
However, the presence of lists makes it hard even for very restricted instances, like bipartite graphs.

A popular approach for such problem is to explore the complexity on restricted instances, in order to understand the boundary of tractability.
One of typical sources of such restricted instances comes from considering graphs excluding certain substructures, most notably, as induced subgraphs.
For a graph $H$, we say that a graph $G$ is \emph{$H$}-free if it does not contain $H$ as an induced subgraph.

The complexity of \lcol{$k$} in $H$-free graphs is quite well-understood. Let us introduce some notation. For an integer $t$, by $P_t$ we denote the $t$-vertex path.
We use `$+$' to denote disjoint union of graphs, so $H_1 + H_2$ denotes the graph with two components, one isomorphic to $H_1$ and the other one to $H_2$.
We also write $rH$ to denote the disjoint union of $r$ copies of $H$.

Classic hardness reductions imply that if $H$ is not a linear forest (i.e., a forest of paths), then already \lcol{3} is \textsf{NP}-hard for $H$-free graphs~\cite{DBLP:journals/cpc/Emden-WeinertHK98,DBLP:journals/siamcomp/Holyer81a,DBLP:journals/jal/LevenG83}.
On the other hand, for every $k$, the \lcol{$k$} problem in $H$-free graphs is polynomial-time solvable if (a) $H$ is an induced subgraph of $rK_3$, for some $r$~\cite{DBLP:journals/corr/abs-2505-00412,DBLP:journals/combinatorica/ChudnovskyHS24,DBLP:journals/siamdm/HajebiLS22}, or (b) $H$ is an induced subgraph of $P_5 + rP_1$, for some $r$~\cite{DBLP:journals/algorithmica/HoangKLSS10,DBLP:journals/algorithmica/0001GKP15}.

For $k \geq 5$, all remaining cases are \textsf{NP}-hard~\cite{DBLP:journals/iandc/GolovachPS14,DBLP:journals/ejc/Huang16,DBLP:journals/algorithmica/0001GKP15}.
The \lcol{4} problem is \textsf{NP}-hard for $P_6$-free graphs~\cite{DBLP:journals/iandc/GolovachPS14,DBLP:journals/ejc/Huang16}, which leaves a number of open cases for disconnected $H$.

The case of \lcol{3} is much more elusive. 
It is known to be in \textsf{P} for $P_7$-free~\cite{DBLP:journals/combinatorica/BonomoCMSSZ18} and in $(P_6+rP_3)$-free graphs~\cite{DBLP:journals/algorithmica/ChudnovskyHSZ21}.
The general belief in the community is that \lcol{3} in $H$-free graphs should solvable in polynomial time for any linear forest $H$.
This belief is supported by the existence of a \emph{quasipolynomial-time} algorithm for all that cases~\cite{DBLP:conf/sosa/PilipczukPR21}.
This is a strong indication that the problem is not \textsf{NP}-hard. However, we seem to be very far from improving the quasipolynomial algorithm to a polynomial one.

Motivated by the notorious open case of $k=3$, Hajebi, Li, and Spirkl~\cite{DBLP:journals/siamdm/HajebiLS24} considered a slightly different setting.
An \emph{ordered graph} is a graph with a fixed linear order of vertices. For ordered graphs $G$ and $H$, we say that $H$ is an induced subgraph of $G$ if it can be obtained from $G$ by deleting vertices. Thus, the relative ordering of vertices of $H$ must coincide with the relative ordering of their images in $G$.
Now, the notion of $H$-free (ordered) graphs is a natural one: we forbid $H$ as an induced (ordered) subgraph.

We remark that this setting allows us to understand the distinction between easy and hard cases even better.
Indeed, excluding $H$ as an unordered induced subgraph is equivalent to excluding all possible orderings of $H$.
But what if we exclude just one, or few possible orderings?

As a motivating example, consider the case that $H = P_3$.
The class of unordered $P_3$-free graphs is very simple: every connected component of such a graph is a clique.
Consequently, \lcol{$k$} is polynomial-time-solvable for every $k$.

Now let us look at the ordered setting. Up to isomorphism, there are three possible orderings of $P_3$, i.e., $\jj$, $\jjrev$, and $\pthree$.
It turns out that forbidding at least one of $\jj$, $\jjrev$ leads to a problem that is in \textsf{P} for every fixed $k$.
(As observed by Hajebi et al.~\cite{DBLP:journals/siamdm/HajebiLS24}, the instances of such a problem are chordal, see~\cref{prop:chordal}).
However, \lcol{$k$} is \textsf{NP}-hard in $\pthree$-free ordered graphs for all $k \geq 3$~\cite{DBLP:journals/siamdm/HajebiLS24}.

The results obtained by Hajebi et. al~\cite{DBLP:journals/siamdm/HajebiLS24} are listed in \cref{tab:summary}.
Let us make a few comments about implications between results.
First, if $H$ is an unordered graph, then hardness for $H$-free graphs implies hardness when we exclude any ordering of $H$.
Second, since we consider the list problem,
an algorithm for \lcol{$k$} implies an algorithm for \lcol{$(k-1)$} and the hardness for \lcol{$(k-1)$} implies the hardness for \lcol{$k$}.
Third, if $H'$ is an induced subgraph of $H$,
then the algorithm for $H$-free graphs implies the algorithm for $H'$-free graphs,
and hardness for $H'$-free graphs implies hardness for $H$-free graphs.
Finally, if $H'$ is an ordered graph obtained from $H$ by reversing the ordering of vertices, the complexity in $H$-free and $H'$-free graphs is the same.

\begin{table}[t]
    \centering
    \begin{tabular}{|cccc|}
    \hline
         Forbidden $H$  & $k=3$ & $k=4$ & $k \geq 5$\\ \hline
           \picJw & \cellcolor{green!50} \cite{DBLP:journals/siamdm/HajebiLS24} & \cellcolor{green!50}  & \cellcolor{green!50} \cref{thm:poly-oneedge} \\
          \picChordal & \cellcolor{green!50}  & \cellcolor{green!50}  & \cellcolor{green!50} \cite{DBLP:journals/siamdm/HajebiLS24} $+$ folklore \\
          \picChordalx & \cellcolor{green!50} \cite{DBLP:journals/siamdm/HajebiLS24} & \cellcolor{green!50} \cref{thm:4colorsj16} & ? \\
         \picChordaly & \cellcolor{green!50} \cite{DBLP:journals/siamdm/HajebiLS24} & \cellcolor{red!50} \cref{thm:jp} & \cellcolor{red!50}  \\
          \mseven & ? & \cellcolor{red!50} \cref{thm:m7} & \cellcolor{red!50}  \\
          \meight & ? & \cellcolor{red!50}  & \cellcolor{red!50}  \\
          \msix & ? & ? & ? \\
          other & \cellcolor{red!50} \cite{DBLP:journals/siamdm/HajebiLS24} & \cellcolor{red!50}  & \cellcolor{red!50}  \\
         \hline
    \end{tabular}
    \caption{Complexity of \lcol{$k$} for $H$-free ordered graphs: state of the art. Green/red cells indicate that the problem is in \textsf{P}/\textsf{NP}-hard. Empty dots indicate vertices that might, but do not have to exist (the number of such vertices is arbitrary).}
    \label{tab:summary}
\end{table}

In this paper we consider the complexity of \lcol{$k$} for $k \geq 4$ in ordered $H$-free graphs.
Let us briefly discuss our results, see also \cref{tab:summary}.

We show that if $H$ has one edge, then \lcol{$k$} is polynomial-time-solvable in $H$-free graphs, for any fixed $k$; see \cref{thm:poly-oneedge}.
This extends earlier result of Hajebi et al.~\cite{DBLP:journals/siamdm/HajebiLS24} for $k=3$.
Even though our algorithm works for every $k$, it is significantly simpler and has better complexity bound.
Our approach is based on branching which allows us to obtain a family of ``cleaned'' instances that can be efficiently solved by dynamic programming.

We also show that \lcol{4} can be solved in polynomial time for graphs that exclude a graph consisting of a copy of $\jj$ preceded by an arbitrary number of isolated vertices; see \cref{thm:4colorsj16}. This algorithm is much more involved than the one from \cref{thm:poly-oneedge}  and is the main algorithmic contribution of the paper.
It consists of a few phases of branching that again lead us to a ``cleaned'' instance. In such an instance we remove certain edges.
We remark that in general such an operation is not safe in $H$-free graphs as it might take us outside the class.
However, we argue that in our case we can indeed do it.
Finally, we reduce the problem to solving \lcol{4} for a chordal graph, which can be done in polynomial time.

In stark contrast, in \cref{thm:jp} we show that if $H$ is a graph consisting of a copy of $\jj$, \emph{followed} by a single isolated vertex, then \lcol{4} is \textsf{NP}-hard in $H$-free graphs.
We remark that this is also where the complexity of \lcol{4} differs from the complexity of \lcol{3}. Indeed, the latter problem is in \textsf{P} in classes defined by forbidding any graph obtained by adding isolated vertices before and after a copy of $\jj$.

Finally, in \cref{thm:m7} we show that \lcol{4} is \textsf{NP}-hard for $\rainbow$-free graphs; this is one of the open cases for \lcol{3}.

We also look at the closely related \listcol problem, where the number of colors is not bounded; see \cref{thm:unbounded}. Unsurprisingly, this problem turns out to be \textsf{NP}-hard in $H$-free graphs for all graphs $H$ with at least three vertices. If $H$ has two vertices, the problem is actually trivial.

The paper is concluded with several open questions and possible directions for further research.

\section{Preliminaries}
For a positive integer $n$, by $[n]$ we denote the set $\{1,\ldots,n\}$.
For two positive integers $i,j$ such that $i<j$ by $[i,j]$ we denote $\{i,i+1,\ldots,j\}$.
For a set $X$, by $2^X$ we denote the set of all subsets of $X$, and by $\binom{X}{k}$, where $k \in N$, we denote the set of all $k$-element subsets of $X$.

For $k,\ell\in \N$, by $\Ram(k,\ell)$ we denote the \emph{Ramsey number of $k$ and $\ell$}, i.e., the minimum number $n$ such that every graph on $n$ vertices contains either a clique on $k$ vertices or an independent set on $\ell$ vertices. It is known that for every $k,\ell\in \N$, the number $\Ram(k,\ell)$ exists.

\paragraph{Ordered graphs.} An ordered graph $G$ is a graph given with a linear ordering of its vertices.
For two vertices $u,v$, we write $u \prec v$ if $u$ appears the ordering earlier than $v$.

We say that $u$ is a \emph{forward neighbor} (resp., \emph{backward neighbor}) of $v$ if $uv \in E(G)$ and 
$v \prec u$ (resp., $u \prec v$).
The set of forward (resp., backward) neighbors of $v$ is denoted by $N^+(v)$ (resp., $N^-(v)$).

\paragraph{Special graphs.} Let us define some special ordered graphs that will be crucial for us.
\begin{description}
\item[\jwp] For a non-negative integer $\ell$, the vertex set of \jwp consists of $\ell+2$ vertices $v_1,\ldots,v_{\ell+2}$, ordered so that $v_i\prec v_j$ for $i<j$, and the edge set consists of one edge $v_{1}v_{\ell+2}$.
\item[\jw] For a non-negative integer $\ell$, the vertex set of \jw consists of $3\ell+2$ vertices $v_1,\ldots,v_{3\ell+2}$, ordered so that $v_i\prec v_j$ for $i<j$, and the edge set consists of one edge $v_{\ell+1}v_{2\ell+2}$.
\item [\jj] The vertex set of \jj consists of three vertices $v_1,v_2,v_3$ ordered so that $v_1\prec v_2\prec v_3$, and the edge set consists of two edges $v_1v_2,v_1v_3$.
\item [$\jp$] The vertex set of $\jp$ consists of four vertices $v_1,v_2,v_3,v_4$ ordered so that $v_1\prec v_2\prec v_3\prec v_4$, and the edge set consists of two edges $v_1v_2,v_1v_3$.
\item[$\lj$] For a non-negative integer $\ell$, the vertex set of $\lj$ consists of $\ell+3$ vertices $v_1,\ldots,v_{\ell+3}$ ordered so that ordered so that $v_i\prec v_j$ for $i<j$, and the edge set consists of two edges $v_{\ell+1}v_{\ell+2},v_{\ell+1}v_{\ell+3}$.
\item [\rainbow] The vertex set of \rainbow consists of four vertices $v_1,v_2,v_3,v_4$, ordered so that $v_1\prec v_2\prec v_3\prec v_4$, and the edge set consists of two edges $v_1v_4,v_2v_3$.
\end{description}

 A graph $G$ is chordal if it contains no cycle on at least $4$ vertices as an induced subgraph. The following observation by Hajebi et al.~\cite{DBLP:journals/siamdm/HajebiLS24} will be useful.
\begin{proposition}[Hajebi et al.~\cite{DBLP:journals/siamdm/HajebiLS24}]\label{prop:chordal}
 A graph $G$ is \emph{chordal} if and only if there exists a linear ordering $\prec$ of $V(G)$ such that $G$ (as an ordered graph) is $\jj$-free.
\end{proposition}

\paragraph{List coloring.} 
The instance of the \listcol problem is $(G,L)$, where $G$ is a graph and $L : V(G) \to 2^{\N}$ is a \emph{list function}.
We ask whether $G$ admits a proper coloring such that every vertex receives a color from its list.

For a fixed integer $k$, the \lcol{$k$} is a restriction of \listcol, where
every list is a subset of~$[k]$. 

The following result of Edwards~\cite{DBLP:journals/tcs/Edwards86} will be useful for us.

\begin{theorem}[Edwards~\cite{DBLP:journals/tcs/Edwards86}]\label{thm:edwards}
 For every $k\in \N$, there is a polynomial-time algorithm that solves every instance $(G,L)$ of \lcol{$k$} such that for every $v\in V(G)$, it holds $|L(v)|\leq 2$.
\end{theorem}

\section{Polynomial-time algorithms}
In both algorithm presented in this section we will use two simple reduction rules.
Let $(G,L)$ be an instance of \lcol{$k$} for some $k$.
\begin{enumerate}[(R1)]
    \item If there is $v\in V(G)$ such that $L(v)=\emptyset$, then return \textsf{NO}.
    \item For $v\in V(G)$ such that $L(v)=\{a\}$, remove $a$ from lists of all neighbors of $v$, and remove $v$ from $G$.    
\end{enumerate}
Clearly, the reduction rules are safe and their exhaustive application can be performed in polynomial time.

\subsection{$H$-free graphs if $|E(H)|=1$}

In this section we prove the following result.
\begin{theorem}\label{thm:poly-oneedge}
Let $H$ be a fixed ordered graph with one edge.
For every $k$, the \lcol{$k$} problem is polynomial-time solvable in $H$-free ordered graphs.
\end{theorem}
\begin{proof}
Observe that every graph $H$ with one edge is an induced subgraph of \jw, for some $\ell \leq |V(H)|$. Thus, it is sufficient to consider the case that $H = \jw$ . We can also assume that $\ell \geq 1$, as for $\ell=0$ the problem is trivial: all instances are edgeless.

We proceed by induction on $k$; the case $k \leq 2$ is obvious.
Thus suppose that $k \geq 3$ and the claim holds for $k-1$.
Let $(G,L)$ be an $n$-vertex instance of \lcol{$k$} such that $G$ is \jw-free ordered graph.

\paragraph{Branching.}
We can first verify if there is a list coloring $c: (G,L)\to [k]$ such that some color $i\in [k]$ is used on at most $2\ell-1$ vertices; for every $i\in [k]$, and for every $A_i\subseteq V(G)$ of size at most $2\ell-1$, we create an instance $(G_1,L_1)$ from $(G,L)$ as follows:
\begin{enumerate}
    \item for every $v\in A_i$, we remove from $L(v)$ all colors but $i$,
    \item for every $v\in V(G)\setminus A_i$, we set its list to $L(v) \setminus \{i\}$,
    \item we exhaustively apply reduction rules (R1) and (R2).
\end{enumerate}
Note that now $(G_1,L_1)$ can be seen as an instance of \lcol{$(k-1)$}, and thus, it can be solved in polynomial time by the inductive assumption.

So since now we can assume that for every list coloring $c: (G,L)\to [k]$, every color $i\in [k]$ is used on at least $2\ell$ vertices.
Now we branch on the choice of first and last $\ell$ vertices in each color, i.e., for every $(2k)$-tuple $\mathsf{A}=(A_1,\ldots,A_k,B_1,\ldots,B_k)$ of pairwise disjoint sets, each of size $\ell$, where $A_i$ precedes $B_i$, we create from $(G,L)$ an instance $(G_{\mathsf{A}},L_{\mathsf{A}})$ as follows:
\begin{enumerate}
    \item for every $i\in [k]$, for $v\in A_i\cup B_i$, we remove from $L(v)$ all colors but $i$,
    \item for every $i\in [k]$, for every $v\in V(G)\setminus (A_i\cup B_i)$, if $v$ precedes the last vertex of $A_i$ or the first vertex of $B_i$ precedes $v$, then we remove $i$ from $L(v)$,
    \item we exhaustively apply reduction rules.
\end{enumerate}
Consider such an instance $(G_{\mathsf{A}},L_{\mathsf{A}})$.
Let $i\in [k]$, and let $X_i$ be the set of vertices $v$ in $G_{\mathsf{A}}$ with $i\in L_{\mathsf{A}}(v)$.

\begin{claim}\label{clm:xi-bridge-free}
For every $i\in [k]$, the graph $G_{\mathsf{A}}[X_i]$ is \jwp-free.
\end{claim}
\begin{claimproof}
Observe that $X_i$ is non-adjacent (in $G$) to $A_i\cup B_i$ since $i$ was not removed from the lists of $X_i$ by (R2).
Moreover, by the definition of $(G_{\mathsf{A}},L_{\mathsf{A}})$, the vertices of $A_i$ precede the vertices of $X_i$, which in turn precede the vertices of $B_i$.
Furthermore, since (R1) did not return \textsf{NO}, we have that the set $A_i\cup B_i$ is independent.
Therefore, if the graph $G_{\mathsf{A}}[X_i]$ contains an induced copy of \jwp, then $G$ contains an induced copy of \jw, a contradiction.
\end{claimproof}

\paragraph{Dynamic programming.} So since now, we assume that we are dealing with an instance  that satisfies the property from \cref{clm:xi-bridge-free}.
For simplicity, let us denote the current instance $(G,L)$, omitting  the subscript $\mathsf{A}$.

Let $v_1 \prec \ldots \prec v_n$ denote the vertices of $G$.
Moreover, for $j\in [n]$, we define $V_j=\{v_1,\ldots,v_j\}$.

Let $j \in [n]$, let $\mathcal{C}_j$ denote the family of all $k$-tuples of pairwise disjoint independent subsets of $V_j$, each of size at most $\ell$.
We say that a list coloring $c: (G[V_j],L)\to [k]$ is \emph{compatible with $(C_1,\ldots,C_k) \in \mathcal{C}_j$} if, for every $i\in [k]$, the following hold:
if $|C_i|\leq \ell-1$, then $c^{-1}(i)=C_i$,
and if $|C_i|=\ell$, then $C_i$ is the set of $\ell$ last vertices of $V_j$ colored with $i$.

For every $j\in [n]$, we will construct a table  $\mathsf{Tab}_j$ of entries, indexed by all elements of $\mathcal{C}_j$; note that the total size of of all tables is bounded by $n \cdot n^{kw}$.
We will set  $\mathsf{Tab}_j[(C_1,\ldots,C_k)]$ to true if and only if there exists a list coloring of $(G[V_j],L)$ that is compatible with $(C_1,\ldots,C_k)$.
Clearly, $(G,L)$ is a yes-instance of \lcol{$k$} if and only if $\mathsf{Tab}_n$ contains at least one true entry.
So it remains to show how to compute the entries efficiently.

First, we set $\mathsf{Tab}_1[C_1,\ldots,C_k]$ to true if and only if there is $i\in L(v_1)$ such that $C_i=\{v_1\}$ and $C_t=\emptyset$ for $t\neq i$.
So now assume that $j \geq 2$, we computed all entries for $\mathsf{Tab}_{j-1}$, and let $(C_1,\ldots,C_k) \in \mathcal{C}_j$.
If there is no $t\in L(v_{j})$ such that $v_{j}\in C_t$, we set $\mathsf{Tab}_{j}[C_1,\ldots,C_k]$ to false.
So now assume that such $t$ exists; clearly it is unique.
We set $\mathsf{Tab}_{j}[(C_1,\ldots,C_k)]$ to true if and only if at least one of following conditions is satisfied.
\begin{enumerate}
    \item $\mathsf{Tab}_{j-1}[(C_1,\ldots,C_{t}\setminus \{v_j\},\ldots,C_k)]$ is true.
    \item $|C_t|=\ell$ and there exists $x\in V_{j-1}$ such that: (i) $x$ precedes all the vertices of $C_t$, (ii) $x$ is non-adjacent to $C_t$, and (iii) for $C'_t=C_t\setminus \{v_j\}\cup \{x\}$, we have that $\mathsf{Tab}_{j-1}[(C_1,\ldots,C_{t-1},C'_{t},C_{t+1},\ldots,C_k]$ is true.
\end{enumerate}
Clearly, all table entries can be computed in polynomial time in $n$.

\paragraph{Correctness.} Let us verify that $\mathsf{Tab}_j$ satisfies the desired conditions.
We proceed by induction on $j\in [n]$.
For $j=1$, the conditions are clearly satisfied.
So now assume that $j \geq 2$ and the table $\mathsf{Tab}_{j-1}$ is filled properly.
Let $(C_1,\ldots,C_k) \in \mathcal{C}_j$.
Since, for each $i \in [k]$, the set $C_i$ is meant to contain the last $|C_i|$ vertices colored with $i$, then there always must be $t\in L(v_j)$ such that $v_j\in C_t$ as $v_j$ has to receive some color from $L(v_j)$ and it is the last vertex in $V_j$.

Suppose first that there exists a list coloring $c$ of $V_j$ that is compatible with $(C_1,\ldots,C_k)$.
If $|c^{-1}(t)|\leq \ell$, then $t$ is used by $c$ precisely $|c^{-1}(t)|-1<\ell$ times on $V_{j-1}$, and thus $\mathsf{Tab}_{j-1}[(C_1,\ldots,C_t\setminus\{v_j\},\ldots,C_k)]$ is true, as this value is certified by $c$ restricted to $V_{j-1}$.
In this case, we set $\mathsf{Tab}_j[(C_1,\ldots,C_t,\ldots,C_k)]$ to true, as desired.
If $|c^{-1}(t)|\geq \ell+1$, then $t$ is used by $c$ at least $\ell$ times on $V_{j-1}$.
Thus, for the set $C'_t$ of last $\ell$ vertices of $V_{j-1}$ colored with $t$ by $c$, the entry $\mathsf{Tab}_j[(C_1,\ldots,C'_t,\ldots,C_k)]$ is set to true.
Since $C'_t$ must be of the form $C'_t=C_t\setminus \{v_j\}\cup \{x\}$ for some $x\in V_{j-1}$, which precedes all vertices of $C_t$ and is non-adjacent to $C_t$, we set $\mathsf{Tab}_j[C_1,\ldots,C_t,\ldots,C_k]$ to true, as desired.

So now suppose that $\mathsf{Tab}_j[(C_1,\ldots,C_k)]$ is set to true, and let $t\in L(v_j)$ be such that $v_j\in C_t$.

Consider two cases, corresponding to two conditions in the algorithm.

\subparagraph{Case 1. $\mathsf{Tab}_{j-1}[(C_1,\ldots,C_t\setminus\{v_j\},\ldots,C_k)]$ is true.}
    By the inductive assumption, there exists a list coloring $c$ of $(G[V_{j-1}],L)$ compatible with $(C_1,\ldots,C_t\setminus\{v_j\},\ldots,C_k)$.
    Note that since $|C_t|\leq \ell$, we have $|C_t\setminus \{v_j\}|\leq \ell-1$, so $t$ is used fewer than $\ell$ times on $V_{j-1}$ and thus $C_t\setminus \{v_j\}$ is the set of all vertices of $V_{j-1}$ colored with $t$.
    Since we assumed that $C_t$ is an independent set, we can safely extend $c$ to $G[V_j]$ by setting $c(v_j)=t$.

\subparagraph{Case 2. $\mathsf{Tab}_{j-1}[(C_1,\ldots,C_t\setminus\{v_j\},\ldots,C_k)]$ is false.}
    Since we set  $\mathsf{Tab}_j[(C_1,\ldots,C_k)]$  to true, this means that  $|C_t|=\ell$ and there exists $x\in V_{j-1}$ which precedes all vertices of $C_t$ and is non-adjacent to $C_t$, and for $C'_t=C_\ell\setminus \{v_j\}\cup \{x\}$, the entry $\mathsf{Tab}_{j-1}[(C_1,\ldots,C'_{t},\ldots,C_k)]$ is  true.
    
    By the inductive assumption there exists a list coloring $c$ of $V_{j-1}$ compatible with $(C_1,\ldots,C'_{t},\ldots,C_k)$.
    We claim that we can extend $c$ to $G[V_j]$ by setting $c(v_j)=t$.

    Suppose it is not possible, i.e., there is $y\in V_{j-1}$ adjacent to $v_j$ such that $c(y)=t$.
    Since $C_t$ is an independent set and $x$ is non-adjacent to $C_t$ (so in particular to $v_j$), we have that $y\notin C'_t$.
    As $C'_t$ is the set of $\ell$ last vertices of $V_{j-1}$ colored with $t$, the vertex $y$ must precede all vertices of $C'_t$.
    Furthermore, since all the vertices of $C'_t\cup \{y\}$ received the same color, the set $C'_t \cup \{y\}$ must be independent.
    Therefore, the set $\{y\}\cup C'_t \cup \{v_j\}$ induces a copy of \jwp.
    Finally, since all the vertices of $\{y\}\cup C'_t \cup \{v_j\}$ contain $t$ on their list, they are all in $X_t$, which contradicts \cref{clm:xi-bridge-free}.
This completes the proof.
\end{proof}

\subsection[$J_{16}$-free graphs]{$\lj$-free graphs}
\begin{theorem}\label{thm:4colorsj16}
    For every fixed $\ell$, \lcol{4} can be solved in polynomial time on $\lj$-free ordered graphs.
\end{theorem}
\begin{proof}
Let $G$ be an ordered $n$-vertex $\lj$-free graph, given with a list function $L: V(G)\to 2^{[4]}$.
We start with checking if $G$ contains a $K_5$ and, if so, we immediately reject the instance.
This can clearly be done in polynomial time by brute force.

First, suppose there is a list 4-coloring of $G$ such that for some $i \in [4]$, at most $\ell-1$ vertices receive color $i$.
We might exhaustively guess such an $i$ and the set of vertices colored $i$, remove them from the graph, and remove $i$ from the lists of all remaining vertices. This way we reduced the problem to solving a polynomial number of instances of \lcol{3}, each of which can be solved in polynomial time by using the result of Hajebi et al.~\cite[Theorem 22]{DBLP:journals/siamdm/HajebiLS24}, mentioned in \cref{tab:summary}.

So, from now on let us focus on looking for a list 4-coloring, where each color appears at least $\ell$ times. The algorithm consists of four main phases.

\paragraph{Phase 1.} For each $i \in [4]$, we exhaustively guess the first $\ell$ vertices that will receive color $i$. More precisely, for every 4-tuple $\mathcal{A}=(A_1,A_2,A_3,A_4)$ of pairwise disjoint independent sets, each of size $\ell$, we create an instance as follows.
    \begin{enumerate}
        \item For every $i \in [4]$ and every vertex $v\in A_i$, we set the list of $v$ to $L(v) \cap \{i\}$.
        \item For every $i\in[4]$, we remove $i$ from lists of all vertices not in $A_i$ that precede the last vertex of $A_i$.
        \item We exhaustively apply reduction rules.
    \end{enumerate}

Observe that as we assumed that in any list 4-coloring of $G$ there are at least $\ell$ vertices in each color, $(G,L)$ is a yes-instance if and only if at least one of the created instances is a yes-instance.
Furthermore, the number of branches in Phase 1 is at most $n^{4\ell}$, which is polynomial.
Consider one such branch for $\mathcal{A}=(A_1,A_2,A_3,A_4)$ and let $(G_1,L_1)$ be the current instance.
Before we proceed to Phase 2, let us analyze the properties of $(G_1,L_1)$.

\begin{claim}\label{claim:common-color}
 Let $x,y,z\in V(G_1)$ be such that $\{x,y,z\}$ induces a $\jj$, where $x \prec y \prec z$.
 Then $L_1(x)\cap L_1(y)\cap L_1(z)=\emptyset$.   
\end{claim}
\begin{claimproof}
    Suppose that there is $i\in[4]$ such that $i\in L_1(x)\cap L_1(y)\cap L_1(z)$.
    Note that then there are no edges between $A_i$ and $x,y,z$, as otherwise (R2) would remove $i$ from lists of $x,y,z$.
    Moreover all vertices of $A_i$ precede $x,y,z$ as otherwise $i$ would also be removed from the lists.
    Thus, $A_i\cup\{x,y,z\}$ induces a $\lj$ in $G$, a contradiction.
\end{claimproof}

Let us classify forward neighbors of a vertex $v$.
We say that a forward neighbor $u$ of $v$ is  \emph{safe (in $(G_1,L_1)$)} if $L_1(u)\cap L_1(v)
\neq \emptyset$, and otherwise it is \emph{dangerous}.
Now, for every vertex we can bound the number of its safe forward neighbors.

\begin{claim}\label{claim:good-neighbors}
Let $v\in V(G_1)$ and let $i \in L_1(v)$.
Then $v$ has at most $3$ forward neighbors whose list contains $i$.
Consequently, $v$ has at most $12$ safe forward neighbors.
\end{claim}
\begin{claimproof}
    Fix $i \in L_1(v)$.
    By \cref{claim:common-color}, for any two forward neighbors $u,w$ of $v$ with $i \in L_1(u) \cap L_1(w)$, we have $uw\in E(G_1)$.

    On the other hand, as $G$ -- and thus $G_1$ as an induced subgraph of $G$ -- is $K_5$-free, there is no $K_4$ in the neighborhood of $v$.
    Therefore, $v$ can have at most three forward neighbors whose list contains $i$.

    Summing up over all choices of $i$, we conclude that the total number of forward neighbors $u$ of $v$ such that $L_1(u)\cap L_1(v)\neq \emptyset$ is at most $12$, which completes the proof of the claim.
\end{claimproof}

We say that a vertex $v$ is \emph{bad} if it has at least $3\ell+4$ dangerous forward neighbors.
If $v$ is not bad, then it is \emph{good}.
Let $\mathsf{Bad}$ and $\mathsf{Good}$ denote, respectively, the sets of bad and good vertices.

Let us make a few comments about bad and good vertices.
First, every $v \in \mathsf{Bad}$ has list of size $2$.
Indeed, the reduction rules assert that there are no vertices with lists of size at most one, and if the list of $v$ has at least three elements, then there are no vertices with lists disjoint with $L_1(v)$ (and in particular $v$ has no dangerous forward neighbors).

Second, for analogous reasons, all dangerous forward neighbors of a bad vertex $v$ have the same list, i.e., $[4] \setminus L_1(v)$.
 
Finally, if $v$ is good, \cref{claim:good-neighbors} implies that $v$ has at most $12+(3\ell+3)=3\ell+15$ forward neighbors.

\begin{claim}\label{claim:bad-non-neighbors}
    Let $v$ be a bad vertex.
    Let $B$ be the set of these $u \in \mathsf{Bad}$ with $u\prec v$ for which $L_1(u)\neq L_1(v)$.
    Then $v$ is non-adjacent to at most $\Ram(5,\ell)-1$ vertices of $B$.
\end{claim}
\begin{claimproof}
    Suppose there is a subset $B'\subseteq B$ such that $v$ is non-adjacent to $B'$ and $|B'|=\Ram(5,\ell)$.
    
    Since $v$ is bad, it has at least $3\ell+4$ dangerous forward neighbors, and all of them have list $K = [4] \setminus L_1(v)$.    
    Since $G$ is $K_5$-free, there is an independent set $B'' \subseteq B'$ of size $\ell$.
    
    Note that for any $u\in B''$, we have $L_1(u) \cap K \neq \emptyset$ since $L_1(u)\neq L_1(v)$ and both $L_1(u)$ and $K$ are of size $2$.    
    Therefore, by \cref{claim:good-neighbors}, $u$ can have at most $3$ forward vertices with list $K$.
    Summing up over all vertices of $B''$, the neighborhood of $B''$ among dangerous forward neighbors of $v$ is of size at most $3\ell$.
    Consequently, there exists a set $X$ of $4$ forward neighbors of $v$ non-adjacent to any vertex of $B''$ (see \cref{fig:algoJ16}). 
    Since we do not have $K_5$ in $G$, there is no $K_4$ in $X$, and thus there is at least one pair of vertices $x,y\in X$ such that $xy\notin E(G)$.
    But then $B''\cup \{v,x,y\}$ induces a copy of $\lj$ in $G$, a contradiction.
\end{claimproof}

\begin{center}
\begin{figure}[t]
    \centering
   \begin{tikzpicture}[every node/.style={draw,circle,fill=white,inner sep=0pt,minimum size=8pt},every loop/.style={}]

\draw (0,-0.5)--(4,-0.5)--(4,0.5)--(0,0.5)--(0,-0.5);
\node[draw=none,fill=none] (b) at (2,0) {$B''$};

\node[fill=blue,opacity=0.5,label=below:\footnotesize{$v$}] (v) at (6,0) {};

\foreach \k in {0,1,2,3,4,5,6}
{
\node[fill=orange,opacity=0.5] (x\k) at (\k+8,0) {};
\draw (v) to [bend left] (x\k);
}

\foreach \k in {0,1,2,3,4}
{
\foreach \j in {0,2,3,6}
{
\draw[color=red] (2,0.5) to [bend left] (x\j);
}
}
\draw[color=red] (2,0.5) to [bend left] (v);

\node[fill=blue,opacity=0.5,label=right:\footnotesize{$\{1,2\}$}] (l1) at (0,2.5) {};
\node[fill=orange,opacity=0.5,label=right:\footnotesize{$\{3,4\}$}] (l1) at (0,2) {};

\end{tikzpicture}
    
     \caption{Bad vertex $v$, its dangerous forward neighbors, and independent set $B''$ of bad vertices which: precede $v$, are non-adjacent to $v$, and have list other than $L(v)$. There are at least $4$ dangerous forward neighbors of $v$ non-adjacent to $B''$. Red edges denote non-edges.}
    \label{fig:algoJ16}
\end{figure}
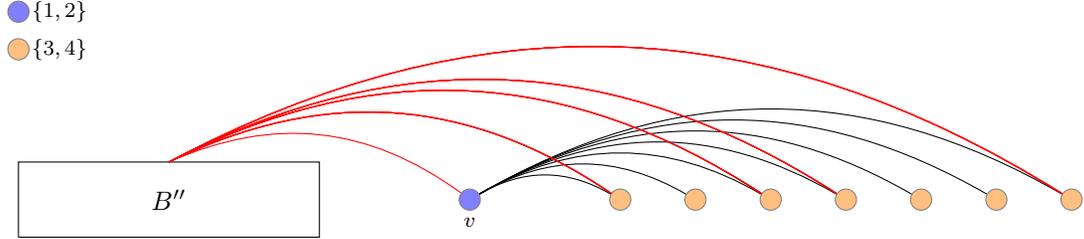
\end{center}

\paragraph{Phase 2.} We proceed further with the instance $(G_1,L_1)$.
We consider two cases depending on the size of $\mathsf{Good}$.
First, if $|\mathsf{Good}| < \Ram(5,\ell)$, then we can exhaustively guess coloring on $\mathsf{Good}$.
More precisely, for each $c: \mathsf{Good} \to [4]$ we create a corresponding instance  by setting the list of every $v\in \mathsf{Good}$ to $L_1(v) \cap \{ c(v)\}$ and then exhaustively applying reduction rules.

The number of instances created this way is at most $4^{\Ram(5,\ell)-1}$, which is a constant as $\ell$ is fixed.
Since every bad vertex has list of size $2$, in any produced instance, every vertex must have list of size $2$.
Here we emphasize that the vertices in the newly created instances might become good, but they were bad before the guessing and their lists could only shrink.
Therefore, applying  \cref{thm:edwards}, we can solve each instance in polynomial time. This completes the proof in this case.

So since now, we can assume that $|\mathsf{Good}| \geq \Ram(5,\ell)$.
Let $D$ be the set of first $\Ram(5,\ell)$ good vertices.
We exhaustively guess the coloring of $D\cup N^+(D)$, i.e., for every $c: D \cup N^+(D) \to [4]$, we create a corresponding instance  by setting the list of every $v \in D\cup N^+(D)$ to $L_1(v) \cap \{c(v)\}$ and applying reduction rules exhaustively. 

Recall that every good vertex has at most $3\ell+15$ forward neighbors, and thus the size of $D\cup N^+(D)$ is bounded by $|D|+|D|\cdot(3\ell+15)=\Ram(5,\ell)\cdot (3\ell+16)$. Therefore, the number of instances created in Phase 2 is at most $4^{\Ram(5,\ell)\cdot (3\ell+16)}$, which is again constant for fixed $\ell$.

Consider one such instance and denote it by $(G_2,L_2)$.
We partition the vertex set of $G_2$ into sets $P,S$ called \emph{prefix} and \emph{suffix} so that in $P$ are all vertices that in $G_1$ precede the last vertex of $D$, and $S$ contains the remaining vertices of $G_2$.
Observe that all vertices in $P$ were bad at the end of Phase 1, i.e., before we guessed a coloring on $D \cup N^+(D)$.
In the instance $(G_2,L_2)$ they do not have to be bad, but we do not care if they are bad now.
We will only use the facts that (i) they have lists of size 2 and (ii)  \cref{claim:bad-non-neighbors} holds for all vertices of $P$. Indeed, note that reduction rules did not increase the size of lists or the number of non-neighbors.

Moreover, observe that the graph induced by $S$ is $\jj$-free.
Indeed, since all vertices of $S$ were not removed by reduction rules applied after the branching in Phase 2, there are no edges between $D$ and $S$.
Moreover, since $G$ has no $K_5$, there is an independent set $D'\subseteq D$ of size $\ell$.
Consequently, if $G[S]$ contained an induced $\jj$, then, together with $D'$, it would form an induced $\lj$ in $G$, a contradiction.
We proceed to the third phase.

\paragraph{Phase 3.} 
Let $(G_2,L_2)$ be an instance produced in Phase 2.
For $X  \in \binom{[4]}{2}$, by $P_X$ we denote the set of vertices of $P$ with list $X$; recall that each vertex of $P$ is bad and thus has list of size $2$.
If, for some $X$, the set $P_X$ does not induce a bipartite graph, we can immediately terminate the current call and reject.
So assume that, for each $X$, the graph $G[P_X]$ is bipartite.
We consider two cases.
    
If $|P_X|< 2\ell$, then we exhaustively guess the coloring of $P_X$, i.e., for every $c: P_X \to X$, we create a corresponding instance by setting the list of every $v \in P_X$ to $L_2(v) \cap \{c(v)\}$ and applying reduction rules exhaustively.
The number of instances created in this case is at most $2^{2\ell-1}$, i.e., a constant.

Now let us assume that $|P_X|\geq 2\ell$.
Denote $X = \{i,j\}$, and suppose that we are dealing with a yes-instance, i.e., there is a coloring $c$ of $G_2$ respecting lists $L_2$.
Observe that  one of the following holds.
\begin{enumerate}[(C1)]
    \item For some $\iota \in \{i,j\}$, at most $\ell-1$ vertices of $P_X$ receive color $\iota$ in $c$.
    \item Each of colors $i,j$ appears at least $\ell$ times on $P_X$ in $c$.
\end{enumerate}
We create the following instances, corresponding to the cases above.
\begin{enumerate}[({I}1)]
\item For every $\iota \in X$ and every set $U \subseteq P_X$ of size at most $\ell-1$, we create a corresponding instance by setting the list of every vertex $v\in U$ to $L_2(v) \cap \{\iota\}$, and the list of every vertex $v\in P_X\setminus U$ to $L_2(v)\setminus \{\iota\}$.
\item For every pair $(U_{X,i},U_{X,j})$, where $U_{X,i},U_{X,j}$ are disjoint independent sets contained in $P_X$, each of size $\ell$, we create a corresponding instance in the following way:
            (i) for every vertex $v\in U_{Xi}$, we set its list to $L_2(v) \cap \{i\}$,
            (ii) for every vertex $v\in U_{X,j}$, we set its list to $L_2(v) \cap \{j\}$, and
            (iii) for every vertex $u$ of $P_X \setminus U_{X,i}$ preceding the last vertex of $U_{X,i}$, we remove $i$ from $L_2(u)$, and
            (iv) and for every vertex $u$ of $P_X\setminus U_{X,j}$ preceding the last vertex of $U_{X,j}$,  we remove $j$ from $L_2(u)$. 
\end{enumerate}  
We exhaustively apply reduction rules to all created instances.
The number of instances created in this case is at most $n^{2\ell}$.

The intended role of the set $U_{X,i}$ (resp., $U_{X,j}$) in (I2) is that it contains the first $\ell$ vertices of $P_X$ colored $i$ (resp., $j$).
If $(G_2,L_2)$ is a yes-instance, and there is coloring satisfying (C1), then at least one of the instances created in (I1) is a yes-instance, and if there is a coloring satisfying (C2), then  at least one of the instances created in (I2) is a yes-instance. 
Thus, if $(G_2,L_2)$ is a yes-instances, then at least one of the instances created in Phase 3 is a yes-instance.
The total number of instances created in this phase is at most $\binom{4}{2} \cdot n^{2\ell}$, i.e., polynomial in $n$.

Let us analyze an instance $(G_3,L_3)$ created in Phase 3.
Recall that all vertices of $P \cap V(G_3)$ have lists of size 2.

\begin{claim}\label{claim:two-lists}
    For any two vertices $v,v' \in P \cap V(G_3)$, either $L_3(v) = L_3(v')$ or $L_3(v) \cap L_3(v') = \emptyset$.
    In particular, at most two distinct lists might appear among the vertices of $P$ in $(G_3,L_3)$.
\end{claim}
\begin{claimproof}
For contradiction suppose otherwise.
By symmetry, assume that $P_{\{1,2\}} \cap V(G_3) \neq \emptyset$ and $P_{\{1,3\}} \cap V(G_3) \neq \emptyset$.

Note that this means that the considered instance was created in a branch corresponding to (I2), both for $X = \{1,2\}$ and for $X = \{1,3\}$.
In particular, we have defined sets $U_{\{1,2\},1} \subseteq P_{\{1,2\}}$ and $U_{\{1,3\},1} \subseteq P_{\{1,3\}}$.
Let $u_{\{1,2\}}$ (resp., $u_{\{1,3\}}$) be the last vertex from $U_{\{1,2\},1}$ (resp.,  $U_{\{1,3\},1}$). 
Note that, since the call was not terminated by (R1), we have $1 \in L_3(u_{\{1,2\}}) \cap L_3(u_{\{1,3\}})$ and these two vertices were removed from the graph by (R2).

Pick any $v_{\{1,2\}} \in P_{\{1,2\}} \cap V(G_3)$ and $v_{\{1,3\}} \in P_{\{1,3\}} \cap V(G_3)$.
As $v_{\{1,2\}}$ was not removed by reduction rules, we have $L_3(v_{\{1,2\}}) = \{1,2\}$.
We claim that $v_{\{1,2\}} \prec u_{\{1,3\}}$.
Suppose otherwise, i.e., all vertices from $U_{\{1,3\}}$ precede $v_{\{1,2\}}$.
By \cref{claim:bad-non-neighbors}, $v_{\{1,2\}}$ is non-adjacent to at most $\Ram(5,\ell)-1$ vertices with list $\{1,3\}$, so in particular, there is $u \in U_{\{1,3\},1}$ which is adjacent to $v_{\{1,2\}}$.
Recall that in the branch defined in (I2) we set the list of $u$ to  $L_3(u) \cap \{1\}$, and since (R1) did not terminate the call, (R2) removed $1$ from the list of $v_{\{1,2\}}$, a contradiction.
Thus, by symmetry, we obtain
\begin{equation} 
    \begin{split} \label{eq:part1}
        v_{\{1,2\}} & \prec u_{\{1,3\}} \\
        v_{\{1,3\}} & \prec u_{\{1,2\}}.
    \end{split}
\end{equation}

Note that if $v_{\{1,2\}} \prec u_{\{1,2\}}$, then condition (iii) or (iv) in (I2) would remove 1 from the list of $v_{\{1,2\}}$. So, again using symmetry, we conclude that
\begin{equation} 
    \begin{split} \label{eq:part2}
        u_{\{1,2\}} & \prec v_{\{1,2\}} \\
        u_{\{1,3\}} & \prec v_{\{1,3\}}.
    \end{split}
\end{equation}
However, inequalities \eqref{eq:part1} and \eqref{eq:part2} are contradictory, so the claim indeed holds.
\end{claimproof}

We proceed to Phase 4.
We emphasize that sets $P$ and $\textsf{Good}$ are not redefined with respect to the current instance.

\paragraph{Phase 4.} Let $(G_3,L_3)$ be an instance given by Phase 3.
By \cref{claim:two-lists}, there are at most two distinct lists on $P\cap V(G_3)$, and by symmetry, we can assume that every vertex of $P$ in $(G_3,L_3)$ has list either $\{1,2\}$ or $\{3,4\}$.
Furthermore, if for $X=\{i,j\}\in \{\{1,2\},\{3,4\}\}$, we have $P_X\cap V(G_3)\neq \emptyset$, then $(G_3,L_3)$ corresponds to a branch of type (I1), i.e., a branch for the sets $U_{X,i},U_{X,j}$.
Let $Y$ be the set of vertices in $G_3$ that consists of safe forward neighbors of $U_{\{1,2\},1}$ (if $P_{\{1,2\}}\cap V(G_3)\neq \emptyset$) and safe forward neighbors of $U_{\{3,4\},3}$ (if $P_{\{3,4\}}\cap V(G_3)\neq \emptyset$).


For every $c: Y \to [4]$, we create a corresponding instance as follows:
\begin{enumerate}
    \item for every $v\in  Y$, we set the list of $v$ to $L_3(v) \cap \{c(v) \}$,
    \item  we exhaustively apply reduction rules,
    \item  we remove all edges whose one endpoint has list $\{1,2\}$ and the other $\{3,4\}$, and at least one endpoint is in $P$.
\end{enumerate}

By \cref{claim:good-neighbors}, each vertex has at most 12 safe forward neighbors and thus, we guess a coloring on at most $ 12\cdot 2\ell=24\ell$ vertices. 
Therefore, the number of instances created in this phase is at most $4^{24\ell}$. 
Note that the edges removed in the final step do not matter for coloring, so we obtain an equivalent instance.
Let $(G_4,L_4)$ be an instance obtained in this phase. 
We point out that in general, if we remove edges in a (ordered) graph which is $F$-free for some graph $F$, it does not have to stay $F$-free.
However, in our case, we will show that the resulting graph is even $\jj$-free.

\begin{claim}
    $G_4$ is $\jj$-free.
\end{claim}
\begin{claimproof}
    Suppose there is an induced $\jj$ on vertices $x,y,z$ with $x \prec y \prec z$.
    Since $G[S]$ is $\jj$-free and we did not remove the edges inside $S$,
    we must have $x\in P$.    
    By symmetry, assume that $L_4(x)=\{1,2\}$.
    
    In particular, $P_{\{1,2\}}\cap V(G_4)\neq \emptyset$.
    As observed before, $(G_4,L_4)$ corresponds to a branch of type (I2) for sets $U_{\{1,2\},1}$, $U_{\{1,2\},2}$.
    By the definition of branch (I2), since the list of $x$ was not reduced to size 1, for every $x' \in U_{\{1,2\},1}$ it holds that $x' \prec x$.
    Since edges $xy, xz$ are not removed, the lists of $y$ and $z$ are not equal to $\{3,4\}$. 
    In particular, if any $v \in U_{\{1,2\},1}$ is adjacent to any $w \in \{x,y,z\}$, then $w$ is a safe forward neighbor of $v$. Consequently, the color of $w$ was guessed and this vertex was removed by reduction rules.
    This means that $U_{\{1,2\},1}$ is non-adjacent to $\{x,y,z\}$.
    Note that $U_{\{1,2\},1}$ was the set of vertices whose lists were all set to $\{1\}$, and (R1) did not return no, so $U_{\{1,2\},1}$ is an independent set of size $\ell$.

    We aim to show that $U_{\{1,2\},1} \cup \{x,y,z\}$ induces $\lj$ in $G$, which would give a contradiction and finish the proof of the claim.
    If this is not the case, it must happen that the edge $yz$ exists in $G_3$, but we removed it in Phase 4. 
    However, this cannot happen as the lists of both $y$ and $z$ are not equal to $\{3,4\}$ and we only removed edges with one endvertex with list $\{3,4\}$. This completes the proof of claim.
\end{claimproof}

Summing up, the instances obtained in Phase 4 are $\jj$-free and thus chordal. Chordal graphs of bounded clique number have bounded treewidth, and thus, each obtained instance can be solved in polynomial time.
This completes the proof.
\end{proof}

Let us conclude this section with brief discussion why the approach from \cref{thm:4colorsj16} fails for more than 4 colors.
In general, in \lcol{$k$}, it is beneficial if adjacent vertices have intersecting lists.
This way deciding on a color of one vertex allows to shrink the list of the other, see e.g.~\cite{DBLP:conf/sosa/PilipczukPR21,DBLP:journals/algorithmica/BonnetR19}.
On the other hand, adjacent vertices with disjoint lists of size at least 2 often appear in hardness proofs of \lcol{k} for $k \geq 4$, see \cref{sec:hardness} or, e.g.,~\cite{DBLP:journals/algorithmica/BonnetR19}.

In the proof of \cref{thm:4colorsj16}, we called (forward) neighbors of $v$ with list disjoint from $L(v)$ \emph{dangerous} and dealing with them was the most technically involved part of the argument.
The crucial property that was very useful here is that every dangerous forward neighbor of $v$ has the same list, i.e., $[4] \setminus L(v)$. This enforces many constraints on possible lists of vertices.

Already for \lcol{5}, a vertex $v$ can have several types of forward neighbors with list disjoint from $L(v)$, and this is the main obstacle in generalizing our approach to more than 4 colors.

\section{Hardness results}\label{sec:hardness}
\subsection[]{$\jp$-free graphs}
\begin{theorem}\label{thm:jp}
 \lcol{4}  is \textsf{NP}-hard on $\jp$-free ordered graphs.
\end{theorem}
\begin{proof}
    Let $\Phi$ be an instance of \sat{3} with variables $v_1,\ldots,v_n$ and clauses $C_1,\ldots,C_m$; we may assume that each clause contains exactly three variables.
    We will construct an ordered graph $G$ with lists $L: V(G)\to 2^{[4]}$ such that $(G,L)$ is a yes-instance of \lcol{4} if and only if $\Phi$ is satisfiable.

    First, for each variable $v_i$ we introduce a vertex $x_i$ with list $\{1,2\}$.
    For an occurrence of a variable $v_i$ in a clause $C_j$ we introduce a vertex $y_{i,j}$ with list $\{3,4\}$.
    Let $X$ denote the set of all vertices of type $x_i$, and $Y$ denote the set of all vertices of type $y_{i,j}$.
    We add all edges between $X$ and $Y$.
    
    Furthermore, for every $i\in [n]$, and every $j\in [m]$ such that $C_j$ contains $v_i$, we introduce vertices $z_{i,j}^1,z_{i,j}^2$ and we add edges $x_iz_{i,j}^1,x_iz_{i,j}^2,y_{i,j}z_{i,j}^1,y_{i,j}z_{i,j}^2$.
    The lists of $z_{i,j}^1,z_{i,j}^2$ are respectively $\{1,4\}$ and $\{2,3\}$ if the occurrence of $v_i$ in $C_j$ is positive, and $\{1,3\}$ and $\{2,4\}$ otherwise. The set of all vertices of type $z_{i,j}^1$ or $z_{i,j}^2$ is denoted by $Z$.
        
    Moreover, for each clause $C_j$, we fix arbitrarily an ordering of its variables. For every $j\in [m]$ we add vertices $a_j,b_j,c_j,d_j$ with lists $L(a_j)=\{1,4\}$, $L(b_j)=\{2,4\}$, $L(c_j)=\{3,4\}$, and $L(d_j)=\{1,2,3\}$. Furthermore, we add edges $a_jd_j,b_jd_j,c_jd_j$, and $a_jy_{i_1,j},b_jy_{i_2,j},c_jy_{i_3,j}$, where $v_{i_1},v_{i_2},v_{i_3}$ are the consecutive variables of $C_j$.
    We denote $C=\{a_j,b_j,c_j,d_j \ | \ j\in[m]\}$.
    This completes the construction of $(G,L)$ (see \cref{fig:hard-J16}).
    Note that $|V(G)|=|X|+|Y|+|Z|+|C|=n+3m+3m+4m=n+10m$, so $|V(G)|=\Oh(n+m)$.

    Now, let us define the ordering of vertices of $G$.
    All vertices of $Z$ precede all vertices of $Y$, which precede all vertices of $C$, which in turn precede all vertices in $X$.
    The ordering within each of the sets is arbitrary, with one exception:
    for any $j\in [m]$, the vertex $d_j$ appears after $a_j,b_j$, and $c_j$.
    This completes the definition of the ordering of $V(G)$ (see \cref{fig:hard-J16}).
    
    \begin{center}
\begin{figure}[t]
    \centering
   \begin{tikzpicture}[every node/.style={draw,circle,fill=white,inner sep=0pt,minimum size=8pt},every loop/.style={}]

\node[fill=yellow,opacity=0.5] (z11) at (0,0) {};
\node[fill=gray,opacity=0.6] (z12) at (1,0) {};
\node[fill=green,opacity=0.5] (z21) at (2,0) {};
\node[fill=red,opacity=0.6] (z22) at (3,0) {};
\node[fill=yellow,opacity=0.5] (z31) at (4,0) {};
\node[fill=gray,opacity=0.6] (z32) at (5,0) {};

\node[fill=orange,opacity=0.5] (y1) at (7,0) {};
\node[fill=orange,opacity=0.5] (y2) at (8,0) {};
\node[fill=orange,opacity=0.5] (y3) at (9,0) {};

\node[fill=green,opacity=0.5] (c1) at (11,0) {};
\node[fill=gray,opacity=0.6] (c2) at (12,0) {};
\node[fill=orange,opacity=0.5] (c3) at (13,0) {};
\node (c4) at (14,0) {};

\node[fill=blue,opacity=0.5] (x1) at (16,0) {};
\node[fill=blue,opacity=0.5] (x2) at (17,0) {};
\node[fill=blue,opacity=0.5] (x3) at (18,0) {};

\foreach \j in {1,2,3}
{
\foreach \k in {1,2,3}
{
\draw (y\j) to [bend right] (x\k);
}
}

\foreach \j in {1,2,3}
{
\draw (z\j1) [bend left] to (x\j);
\draw (z\j1) [bend right] to (y\j);
\draw (z\j2) [bend left] to (x\j);
\draw (z\j2) [bend right] to (y\j);
\draw (y\j) [bend left] to (c\j);
\draw (c\j) [bend left] to (c4);
}

\node[draw=none,fill=none] (Z) at (2.5,-2) {$Z$};
\node[draw=none,fill=none] (Y) at (8,-2) {$Y$};
\node[draw=none,fill=none] (C) at (12.5,-2) {$C$};
\node[draw=none,fill=none] (X) at (17,-2) {$X$};

\node[fill=blue,opacity=0.5,label=right:{$\{1,2\}$}] (l) at (0,4) {};
\node[fill=orange,opacity=0.5,label=right:{$\{3,4\}$}] (l) at (0,3.5) {};
\node[fill=green,opacity=0.5,label=right:{$\{1,4\}$}] (l) at (0,3) {};

\node[fill=yellow,opacity=0.5,label=right:{$\{1,3\}$}] (l) at (1.5,4) {};
\node[fill=red,opacity=0.6,label=right:{$\{2,3\}$}] (l) at (1.5,3.5) {};
\node[fill=gray,opacity=0.6,label=right:{$\{2,4\}$}] (l) at (1.5,3) {};

\node[label=right:{$\{1,2,3\}$}] (l) at (3,4) {};

\end{tikzpicture}
    
     \caption{Construction from \cref{thm:jp} for a clause $(\neg v_1\lor v_2\lor \neg v_3)$.}
    \label{fig:hard-J16}
\end{figure}
\end{center}

    Now let us verify the equivalence of instances.
    
    \begin{claim}
        $(G,L)$ is a yes-instance of \lcol{4}  if and only if $\Phi$ is satisfiable.
    \end{claim}
    \begin{claimproof}
        First assume that there is a satisfying assignment $\psi: \{v_1,\ldots,v_n\} \to \{\textsf{true}, \textsf{false}\}$ of $\Phi$.
        Let us define a 4-coloring $f$ of $G$ as follows.
        For every $i\in [n]$, if $\psi(v_i)=\textsf{true}$, we set $f(x_i)=1$, $f(y_{i,j})=3$ if the occurrence of $v_i$ in $C_j$ is positive, and $f(y_{i,j})=4$ if the occurrence of $v_i$ in $C_j$ is negative.
        Otherwise, we set $f(x_i)=2$, $f(y_{i,j})=4$ if the occurrence of $v_i$ in $C_j$ is positive, and $f(y_{i,j})=3$ if the occurrence of $v_i$ in $C_j$ is negative.
        Note that $f(y_{i,j}) = 3$ if and only if $v_i$ appears in $C_j$ as a true literal.
        So far, we used disjoint sets of colors to color both sides of the biclique induced by $X \cup Y$, so $f$ respects the edges with both endpoints already colored. 
        
        Now consider a vertex $z$ of $Z$ -- note that it has one neighbor $x \in X$ and one neighbor $y \in Y$.
        If $z$ corresponds to a positive occurrence, then $L(z) \in \{ \{1,4\}, \{2,3\} \}$, and the colors of $x,y$ are either $1,3$ or $2,4$.
        Analogously, if $z$ corresponds to a negative occurrence, then $L(z) \in \{ \{1,3\}, \{2,4\} \}$, and the colors of $x,y$ are either $1,4$ or $2,3$.
        In all cases, we can always find a color for $z$ that is different from colors of its neighbors.
                
        It remains to extend $f$ to the vertices of $C$. Let $j\in[m]$. Since $\psi$ is a satisfying assignment of $\phi$, there is at least one true literal in  $C_j$. Therefore, for some $i$, the vertex $y_{i,j}$ is colored with $3$.
        Now depending on the order of $v_i$ among the variables of $C_j$, one of the vertices $a_j,b_j,c_j$ (the one adjacent to $y_{i,j}$) can be colored with $4$, so we can extend $f$ first to $a_j,b_j,c_j$ so that at least one of them is colored with $4$, and thus there will be a color left for $d_j$ whose list is $\{1,2,3\}$. This completes the definition of $f$.

        So now assume there is an coloring $f$ of $G$ that respects $L$.
        We define an assignment $\psi$ of $\Phi$, so that $\psi(v_i)=\textsf{true}$ if and only if $f(x_i)=1$.
        Let us verify that $\psi$ satisfies $\Phi$.
        Consider $j\in [m]$. Since $d_j$ is colored with one of $1,2,3$, at least one of $a_j,b_j,c_j$, whose lists are respectively $\{1,4\},\{2,4\}$ and $\{3,4\}$, has to be colored with $4$, let $u_j\in \{a_j,b_j,c_j\}$ be a vertex colored by $f$ with $4$. The neighbor $y_{i,j}$ of $u_j$ in $Y$, which corresponds to an occurrence of $v_i$ in $C_j$, has to be colored with $3$. 
        Now consider vertices $z_{i,j}^1$ and $z_{i,j}^2$.
        If the occurrence of $v_i$ in $C_j$ is positive, these two vertices from $Z$ have lists $\{1,4\}$ and $\{2,3\}$, and thus we must have $f(x_i)=1$, so $\psi(v_i)=\textsf{true}$ and $v_i$ satisfies $C_j$.
        If the occurrence of $v_i$ in $C_j$ is negative, the two vertices from $Z$ have lists $\{1,3\}$ and $\{2,4\}$, and thus we must have $f(x_i)=2$, so $\psi(v_i)=\textsf{false}$ and $v_i$ satisfies $C_j$.

        This completes the proof of claim.
    \end{claimproof}

I remains to show that $G$ is $\jp$-free.
    \begin{claim}
    $G$ is $\jp$-free.
    \end{claim}
    \begin{claimproof}
        Suppose that $G$ contains an induced subgraph isomorphic to $\jp$, and denote the vertices of this induced subgraph as $w_1,w_2,w_3,w_4$, where $w_1 \prec w_2 \prec w_3 \prec w_4$. Let us analyze where these four vertices might appear in $G$.
        Let us start with $w_1$.
        Note that it cannot be in $C\cup X$, as each of the vertices in $C\cup X$ has at most one forward neighbor.
        Suppose that $w_1 \in Z$. Note that each vertex in $Z$ has degree $2$ and its neighbors are adjacent. Since $w_2,w_3$ are non-adjacent, we observe that $w_1 \notin Z$.
        Consequently, we must have $w_1 \in Y$. Each vertex of $Y$ has one neighbor in $C$ and its remaining forward neighbors are in $X$.
        In particular, $w_4$ must be in $X$. However, this is a contradiction, as $w_1$ and $w_4$ are non-adjacent.
    \end{claimproof}
    
    This completes the proof of the theorem.
\end{proof}
\subsection[]{$\rainbow$-free graphs}
In this section we show the following hardness result.

\begin{theorem}\label{thm:m7}
\lcol{4} is \textsf{NP}-hard on $\rainbow$-free graphs.
\end{theorem}

\subsubsection{The construction}

First, let us describe the  main building blocks in our reduction.
For positive integers $\ell,\ell'$, a \emph{link} is  a tuple $(F,(x_1,\ldots,x_\ell),(y_1,\ldots,y_{\ell'}))$, where $F$ is an ordered graph, and $(x_1,\ldots,x_\ell)$ and $(y_1,\ldots,y_{\ell'})$ are tuples of vertices of $F$, such that:
\begin{enumerate}
    \item $x_1,\ldots,x_\ell$ are, in this order, the first $\ell$ vertices of $F$,
    \item $y_1,\ldots,y_{\ell'}$ are, in this order, the last $\ell'$ vertices of $F$,
    \item sets $\{x_1,\ldots,x_\ell\}$ and $\{y_1\ldots,y_{\ell'}\}$ are disjoint and independent,
    \item $F$ is $\rainbow$-free.    
\end{enumerate}
The vertices $x_1,\ldots,x_\ell$ (resp., $y_1,\ldots,y_{\ell'})$) are called \emph{input} (resp., \emph{output}) vertices of the link.

The properties of links allow us combine them, without creating the forbidden subgraph.
This operation, that we call \emph{chaining}, is formally described in the following lemma.

\begin{lemma}\label{obs:m7-combining}
Let $(F^1,(x^1_1,\ldots,x^1_\ell),(y^1_1,\ldots,y^1_{\ell'}))$
and $(F^2,(x^2_1,\ldots,x^2_{\ell'}),(y^2_1,\ldots,y^2_{\ell''}))$ be two links on disjoint sets of vertices.
The ordered graph $F$ obtained from $F_1$ and $F_2$ by identifying  $y^1_i$ with $x^2_i$ for each $i\in [\ell']$  is a link.
\end{lemma}
\begin{proof}
The first three properties of a link are obvious, so, for contradiction, suppose that $F$ contains an induced copy of $\rainbow$. Let $v_1,v_2,v_3,v_4$ be its consecutive vertices.
Since both $F_1$ and $F_2$ are $\rainbow$-free, it must hold that $v_1\in V(F_1)\setminus V(F_2)$ and $v_4\in V(F_2)\setminus V(F_1)$.
However, $v_1v_4$ is an edge, but, by the construction of $F$, there are no edges between $V(F_1)\setminus V(F_2)$ and $V(F_2)\setminus V(F_1)$, a contradiction.
\end{proof}

In our reduction we will use gadgets that are links enriched by a list function.
We will also speak about chaining gadgets -- in such a situation, we always ensure that the lists of vertices that are identified are equal. Thus, the operation is straightforward.

The main gadget used in our hardness proof is a \emph{NAE gadget}.
\begin{definition}[NAE gadget]
    Let $n \in\N$ and let $I\subseteq \binom{[n]}{3}$ be a set of pairwise disjoint subsets of $[n]$, each of size $3$.
    A \emph{NAE gadget} (for $I$) is a tuple $(C,(x_1,\ldots,x_n),(y_1,\ldots,y_n),L)$, where $(C,(x_1,\ldots,x_n),(y_1,\ldots,y_n))$ is a link and $L:V(C)\to 2^{[4]}$ is a list function, such that:
    \begin{enumerate}[(C1)]
        \item For every $i\in[n]$, it holds that $L(x_i)=L(y_i)=\{1,2\}$.
        \item For every $f: \{x_1,\ldots,x_n\}\to \{1,2\}$ such that for every $\{i,j,k\}\in I$, it holds that $\{f(x_i),f(x_j),f(x_k)\}=\{1,2\}$, we can extend $f$ to a coloring of $(C,L)$.
        \item For every coloring $f$ of $(C,L)$, it holds that:
        \begin{enumerate}
            \item for every $i\in [n]$, $f(x_i)=f(y_i)$,
            \item for every $\{i,j,k\}\in I$, it holds that $\{f(x_i),f(x_j),f(x_k)\}=\{1,2\}$
        \end{enumerate}
    \end{enumerate}    
\end{definition}
For simplicity of notation, we will denote the gadget by $C$.

Intuitively, the gadget plays two roles. First, it transfers the coloring of input vertices to their corresponding output vertices.
Second, it makes sure that for every triple in $I$, the input/output vertices corresponding to that triple are not monochromatic or, in other words, \emph{not all equal} (NAE).
The following lemma is the main technical ingredient of our hardness proof.

\begin{lemma}\label{lem:clause-gadget}
Let $n\in\N$ and let $I\subseteq \binom{[n]}{3}$ be a set of pairwise disjoint subsets of $[n]$, each of size $3$.
In time polynomial in $n$, we can construct a NAE gadget for $I$.
\end{lemma}

Let us postpone the proof of \cref{lem:clause-gadget}, and first let us show \cref{thm:m7} assuming that \cref{lem:clause-gadget} holds.

\begin{proof}[Proof of \cref{thm:m7}]
We reduce from \textsc{Positive NAE-$3$-Sat}. An instance of this problem consists of a set of boolean variables and a set of clauses, each containing exactly three variables (none of them is negated).
We ask if there exists a truth assignment in which each clause contains a true and a false variable, i.e., for each clause, the variables in it are not all equal.

Let $\Phi$ be an instance of \textsc{Positive NAE-$3$-Sat} with variables $v_1,\ldots,v_n$ and clauses $C_1,\ldots,C_m$. We can assume that every variable appears at most four times~\cite{DBLP:journals/corr/abs-1908-04198}.
In polynomial time, we will construct an instance $(G,L)$ of \lcol{4} such that:
\begin{enumerate}
    \item $G$ is an ordered $\rainbow$-free graph,
    \item $(G,L)$ admits a proper coloring if and only if $\Phi$ is satisfiable.
\end{enumerate}

As each clause has three variables and each variable appears in at most four clauses,
we can greedily partition the set of clauses into 10 pairwise disjoint subsets; denote them by $\cC_1,\ldots,\cC_{10}$ (we can think of this as a greedy edge coloring of a hypergraph: for each hyperedge, each its vertex blocks at most three colors). 
For $s\in[10]$, we define $I_s\subseteq \binom{[n]}{3}$ so that $\{i,j,k\}\in I_s$, if and only if there is a clause $\{v_i,v_j,v_k\}$ in $\cC_s$.

Now we are ready to construct the instance $(G,L)$.
We start with introducing NAE gadgets $C_{I_1},\ldots,C_{I_{10}}$, for, respectively, $I_1,\ldots,I_{10}$.
For $s\in [10]$, let us denote the input (resp., output) vertices of $C_{I_s}$ by by $(x_1^s,\ldots,x_n^s)$  (resp., $(y_1^s,\ldots,y_n^s)$).
Next, we chain gadgets $C_{I_1},\ldots,C_{I_{10}}$, i.e., for every $s\in [9]$, we identify the output vertices of $C_{I_s}$ with input vertices of $C_{I_{s+1}}$.
This completes the construction of $(G,L)$.

As each gadget is in particular a link, a repeated application of \cref{obs:m7-combining} yields that $G$ is $\rainbow$-free.
Thus, we are left with proving the equivalence of instances.

Suppose that $\Phi$ is satisfiable and let $\psi:\{v_1,\ldots,v_n\}\to \{\textsf{true},\textsf{false}\}$ be a satisfying assignment.
For every $i\in [n]$, we set $f(x_i^1)=1$ if $\psi(v_i)=\textsf{true}$, and $f(x_i^1)=2$ if $\psi(v_i)=\textsf{false}$.
Since for every clause $\{v_i,v_j,v_k\}\in \cC_1$, we have $\{\psi(v_i),\psi(v_j),\psi(v_k)\} = \{\textsf{true}, \textsf{false} \}$, and thus, for every $\{i,j,k\}\in I_1$ we have $\{f(x_i^1),f(x_j^1),f(x_k^1)\} = \{1,2\}$,
we can extend this coloring to a coloring of the whole gadget $C_{I_1}$.
Since for every $i\in[n]$, we have that $f(x_i^1)=f(y_i^1)=f(x_i^2)$, we can repeat the reasoning inductively for every $s\in [10]$ and extend $f$ to all vertices of $G$.

Now suppose that there is a proper coloring $f$ of $(G,L)$.
We define $\psi: \{v_1,\ldots,v_n\}\to \{\textsf{true},\textsf{false}\}$ so that $\psi(v_i)=\textsf{true}$ if $f(x_i^1)=1$ and $\psi(v_i)=\textsf{false}$ if $f(x_i^1)=2$.
Let us verify that $\psi$ satisfies $\Phi$.
Suppose that there is a clause $C_q = \{v_i,v_j,v_k\}$ which is not satisfied,
i.e., all $v_i,v_j,v_k$ are either set $\textsf{true}$ or set \textsf{false} by $\psi$.
By property (C3)~(a), this means that for every $s\in [10]$, we have  $f(x_i^s)=f(x_j^s)=f(x_k^s)$.
As $C_q \in \cC_s$ for some $s \in [10]$, we obtain a contradiction with property (C3)~(b) of a NAE gadget,
This completes the proof.
\end{proof}

\subsubsection{From NOT-$ccc$ gadgets to NAE gadgets}

We will construct NAE gadgets in several steps.
First, let us show that in order to construct a NAE gadget, it is sufficient to construct its restricted variant, called \emph{NOT-$ccc$ gadget}.

\begin{definition}[NOT-$ccc$ gadget]\label{def:notccc}
    Let $c \in \{1,2\}$.
    Let $n \in\N$ and let $I\subseteq \binom{[n]}{3}$ be a set of pairwise disjoint subsets of $[n]$, each of size $3$.
    A \emph{NOT-$ccc$ gadget} (for $I$) is a tuple $(C^c,(x_1,\ldots,x_n),(y_1,\ldots,y_n),L)$, where $(C^c,(x_1,\ldots,x_n),(y_1,\ldots,y_n))$ is a link and $L:V(C^c)\to 2^{[4]}$ is a list function, such that:
    \begin{enumerate}[(C1)]
        \item For every $i\in[n]$, it holds that $L(x_i)=L(y_i)=\{1,2\}$.
        \item For every $f: \{x_1,\ldots,x_n\}\to \{1,2\}$ such that for every $\{i,j,k\}\in I$, it holds that $\{f(x_i),f(x_j),f(x_k)\} \neq \{c\}$, we can extend $f$ to a coloring of $(C^c,L)$.
        \item For every coloring $f$ of $(C^c,L)$, it holds that:
        \begin{enumerate}
            \item for every $i\in [n]$, $f(x_i)=f(y_i)$,
            \item for every $\{i,j,k\}\in I$, it holds that $\{f(x_i),f(x_j),f(x_k)\} \neq \{c\}$
        \end{enumerate}
    \end{enumerate}    
\end{definition}
As usual, we will simply denote the gadget by $C^c$.

Let us emphasize the difference between a NAE gadget and a NOT-$ccc$ gadget.
The first one is responsible for ensuring that for each triple in $I$, the colors assigned to the corresponding vertices are \emph{not all 1} and \emph{not all 2}.
The role of a NOT-$ccc$ gadget is to ensure that the vertices corresponding to each triple in $I$ are not all colored $c$, where $c \in \{1,2\}$.
Thus, a NAE gadget is simultaneously a NOT-111 gadget and a NOT-222 gadget.

As the gadget not only forbids some colorings, but also copies the coloring of the input to the output,
we immediately obtain the following observation.

\begin{observation}\label{obs:from-nai-to-nae}
    Chaining a NOT-111 gadget and a NOT-222 yields a NAE gadget.
\end{observation}

Thus, in order to prove \cref{lem:clause-gadget}, it is sufficient to build a NOT-$ccc$ gadget for $c \in \{1,2\}$.

\subsubsection{Permutation gadgets}

On the way to the proof of \cref{lem:clause-gadget}, we will first construct another type of a gadget.

\begin{definition}[Permutation gadget]\label{def:permutation-gadget}
Let $\ell\in \N$, and let $\sigma:[\ell]\to [\ell]$ be a permutation.
A \emph{permutation gadget} (for $\sigma$) is a tuple $(P_{\sigma},(x_1,\ldots,x_\ell),(y_1,\ldots,y_\ell),L)$, where $(P_{\sigma},(x_1,\ldots,x_\ell),(y_1,\ldots,y_\ell))$ is a link and $L : V(P_\sigma)\to 2^{[4]}$ is a list function, such that:
\begin{enumerate}[(P1)]
 \item For every $i\in[\ell]$, it holds that $L(x_i)=L(y_i)=\{1,2\}$.
 \item Every mapping $f: \{x_1,\ldots,x_\ell\}\to \{1,2\}$ can be extended to a coloring of $(P_\sigma,L)$.
 \item For every coloring $f$ of $(P_\sigma,L)$, for every $i\in [\ell]$, it holds that $f(x_i)=f(y_{\sigma(i)})$.
\end{enumerate} 
\end{definition}
Again, we will shortly denote a permutation gadget for $\sigma$ by $P_\sigma$.
Intuitively, the role of a permutation gadget is to transfer the coloring of the input to the output, but after applying $\sigma$ on it.
We will show that we can efficiently construct permutation gadgets.

\begin{lemma}\label{lem:permutation-gadget}
Let $\ell\in \N$, and let $\sigma:[\ell]\to [\ell]$ be a permutation.
In time polynomial in $\ell$, we can construct a permutation gadget for $\sigma$.
\end{lemma}

Let us break the proof of \cref{lem:permutation-gadget} in three steps.

First, we observe that  that chaining permutation gadgets yields a permutation gadget for the composition of permutations: this follows directly from \cref{obs:m7-combining} and the definition of a permutation gadget.

\begin{observation}\label{obs:compose-permutation-gadgets}
For $\ell \in \N$, let $\sigma,\sigma' : [\ell] \to [\ell]$ be two permutations and let $P_{\sigma},P_{\sigma'}$ be their corresponding permutation gadgets.
Let $P$ be obtained by chaining $P_{\sigma}$ with $P_{\sigma'}$.
Then, $P$ is a permutation gadget for $\sigma' \circ \sigma$.
\end{observation}

Thus, in order to show \cref{lem:permutation-gadget}, it is enough to construct permutation gadgets for some basic permutations that can be composed into an arbitrary permutation.
These basic permutations are \emph{rotations}.
For integers $j \leq k \leq \ell$, a rotation, denoted by $\langle \ell; j,k\rangle$ is a permutation of $\ell$ such that:
\[
\langle \ell; j,k\rangle(i) = \begin{cases}
i & \text{ if }  i < j \text{ or } i > k\\
i+1 & \text{ if }  j \leq i < k\\
j & \text{ if }  i=k.
\end{cases}
\]
In other words, $\langle \ell; j,k\rangle$ performs a cyclic shift on elements $\{j,\ldots,k\}$, leaving the remaining ones untouched. In particular, if $j=k$, then $\langle \ell; j,k \rangle$ is an identity.

\begin{observation}\label{obs:compose-rotations}
    Every permutation $\sigma : [\ell] \to [\ell]$ can be written as the composition of $\ell-1$ rotations.
\end{observation}
\begin{proof}
    The construction mimics the selection sort algorithm.
    For $i = 1,2,\ldots,\ell-1$ we find, in the current permutation, the element $\sigma(i)$, i.e., the element that should be placed on position $i$ in the final permutation.
    Suppose it is on position $k$.
    Then, we apply the rotation $\langle \ell; i,k \rangle$ bringing $\sigma(i)$ to its desired position.
    Note that after this step, elements on positions $1,\ldots,i$ are already as they should be and we never need to touch them again. Thus, we increase $i$ by 1 and repeat.

    Let us describe it more formally.
    Let $\sigma_0$ be the identity permutation.
    Now, for $i \in [\ell-1]$, we define
    \begin{align*}
        k = \ & (\sigma_{i-1} \circ \ldots \circ \sigma_0\circ \sigma^{-1})(i)\\
        \sigma_i = \ & \langle \ell; i, k \rangle.
    \end{align*}
    
    It is straightforward to verify that $\sigma = \sigma_{\ell-1} \circ \ldots \circ \sigma_1$.
\end{proof}

Thus, in order to prove \cref{lem:permutation-gadget}, it is sufficient to show how to construct permutation gadgets for rotations.
We do this in the next lemma, consult also~\cref{fig:rotation-gadget}.

\begin{center}
\begin{figure}[t]
    \centering
   
    \begin{tikzpicture}[every node/.style={draw,circle,fill=white,inner sep=0pt,minimum size=8pt},every loop/.style={}]


\foreach \k in {1,2,3,4,5,6}
{
\node[fill=blue,opacity=0.5,label=below:{\k}] (s1\k) at (\k*0.5,0) {};
}

\node[fill=blue,opacity=0.5,label=below:{$1$}] (s51) at (15+1*0.5,0) {};
\node[fill=blue,opacity=0.5,label=below:{$2$}] (s52) at (15+2*0.5,0) {};
\node[fill=blue,opacity=0.5,label=below:{$5$}] (s53) at (15+3*0.5,0) {};
\node[fill=blue,opacity=0.5,label=below:{$3$}] (s54) at (15+4*0.5,0) {};
\node[fill=blue,opacity=0.5,label=below:{$4$}] (s55) at (15+5*0.5,0) {};
\node[fill=blue,opacity=0.5,label=below:{$6$}] (s56) at (15+6*0.5,0) {};


\foreach \k in {1,2,3,4,6}
{
\node[fill=blue,opacity=0.5] (s3\k) at (7.5+\k*0.5,0) {};
}
\node[fill=green,opacity=0.5] (s35) at (7.5+5*0.5,0) {};


\foreach \k in {1,2,3,4,7}
{
\node[fill=blue,opacity=0.5] (s2\k) at (3.5+\k*0.5,0) {};
}
\node[fill=yellow,opacity=0.5] (s25) at (3.5+5*0.5,0) {};
\node[fill=orange,opacity=0.5] (s26) at (3.5+6*0.5,0) {};

\foreach \k in {1,2,5,6,7}
{
\node[fill=blue,opacity=0.5] (s4\k) at (11+\k*0.5,0) {};
}

\node[fill=yellow,opacity=0.5] (s43) at (11+3*0.5,0) {};
\node[fill=orange,opacity=0.5] (s44) at (11+4*0.5,0) {};

\draw (s11) to [bend left] (s21);
\draw (s12) to [bend left] (s22);
\draw (s13) to [bend left] (s23);
\draw (s14) to [bend left] (s24);
\draw (s15) to [bend left] (s25);
\draw (s15) to [bend left] (s26);
\draw (s16) to [bend left] (s27);

\draw (s21) to [bend right] (s31);
\draw (s22) to [bend right] (s32);
\draw (s23) to [bend right] (s33);
\draw (s24) to [bend right] (s34);
\draw (s25) to [bend right] (s35);
\draw (s26) to [bend right] (s35);
\draw (s27) to [bend right] (s36);

\draw (s31) to [bend left] (s35);
\draw (s32) to [bend left] (s35);
\draw (s33) to [bend left] (s35);
\draw (s34) to [bend left] (s35);
\draw (s35) to [bend left] (s36);

\draw (s31) to [bend right] (s41);
\draw (s32) to [bend right] (s42);
\draw (s33) to [bend right] (s45);
\draw (s34) to [bend right] (s46);
\draw (s35) to [bend right] (s43);
\draw (s35) to [bend right] (s44);
\draw (s36) to [bend right] (s47);

\draw (s41) to [bend left] (s51);
\draw (s42) to [bend left] (s52);
\draw (s43) to [bend left] (s53);
\draw (s44) to [bend left] (s53);
\draw (s45) to [bend left] (s54);
\draw (s46) to [bend left] (s55);
\draw (s47) to [bend left] (s56);

\node[fill=blue,opacity=0.5,label=right:\footnotesize{$\{1,2\}$}] (l1) at (0.5,2) {};
\node[fill=green,opacity=0.5,label=right:\footnotesize{$\{3,4\}$}] (l2) at (0.5,1.5) {};
\node[fill=yellow,opacity=0.5,label=right:\footnotesize{$\{1,4\}$}] (l3) at (2.5,2) {};
\node[fill=orange,opacity=0.5,label=right:\footnotesize{$\{2,3\}$}] (l4) at (2.5,1.5) {};

\end{tikzpicture}
     \caption{A permutation gadget for the rotation $\langle \ell ; 3,5 \rangle$.}
    \label{fig:rotation-gadget}
\end{figure}
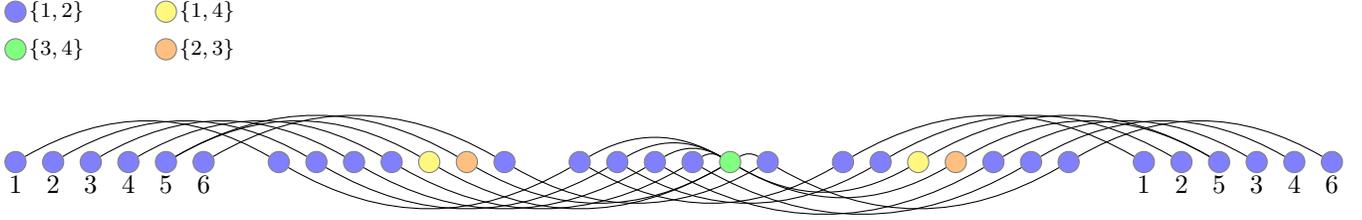
\end{center}

\begin{lemma}\label{lem:rotation-gadget}
Let $j < k \leq \ell\in \N$ be integers.
In time polynomial in $\ell$, we can construct a permutation gadget for $\langle \ell; j,k \rangle$.
\end{lemma}
\begin{proof}

The gadget has $5\ell +2$ vertices. We partition them into five sets $S_1,\ldots,S_5$ such that $|S_1|=|S_3|=|S_5|=\ell$ and $|S_2|=|S_4|=\ell+1$ in a natural way:
the first $\ell$ vertices belong to $S_1$, next $\ell + 1$ vertices belong to $S_2$ and so on.
For $p \in [5]$, the $i$-th vertex of $S_p$ is denoted by $s^p_i$.

The lists of vertices of the gadget are as follows:
\[
    L(s) =
    \begin{cases}
        \{1,4\} & \text{ if } s \in \{s^2_k, s^4_j\} \\
        \{2,3\} & \text{ if } s \in \{s^2_{k+1},s^4_{j+1}\}\\
        \{3,4\} & \text{ if } s = s^3_k \\
        \{1,2\} & \text{ otherwise.}
    \end{cases}
\]

The edge set of the gadget consists of the following elements:
\begin{description}
    \item[Between $S_1$ and $S_2$:] $s^1_is^2_i$ for $i\in [k]$ and $s^1_is^2_{i+1}$ for $i \in [k, \ell]$,
    \item[Between $S_2$ and $S_3$:] $s^2_is^3_i$ for $i\in [k]$ and $s^2_is^3_{i-1}$ for $i \in [k+1,\ell+1]$,
    \item[Inside $S_3$:]  $s^3_ks^3_i$ for $i\in[\ell] \setminus \{k\}$,
    \item[Between $S_3$ and $S_4$:] $s^3_is^4_i$ for $i\in [j-1]$ and $s^3_is^4_{i+2}$ for $i \in [j, k-1]$ and $s^3_ks^4_j,s^3_ks^4_{j+1}$, and $s^3_is^4_{i+1}$ for $i \in [k+1,\ell]$,
    \item[Between $S_4$ and $S_5$:] $s^4_is^5_i$ for $i \in [j]$ and $s^4_is^5_{i-1}$ for $i \in [j+1,\ell+1]$.        
\end{description}

This completes the construction of the gadget $P_{\langle \ell; j,k \rangle}$: the input vertices are $S_1$ and the output vertices are $S_5$.
Clearly,  $S_1$ and $S_5$ are disjoint and independent.
In order to show that the constructed graph is a link, let us verify that it is $\rainbow$-free.

\begin{claim}
    $P_{\langle \ell; j,k \rangle}$ is $\rainbow$-free.
\end{claim}
\begin{claimproof}
For contradiction, suppose there is an induced copy of $\rainbow$ in $P_{\langle \ell; j,k \rangle}$ and let $u_1,u_2,u_3,u_4$ be its consecutive vertices.
Recall that each edge of $P_{\langle \ell; j,k \rangle}$ either connects a vertex of $S_i$ with a vertex of $S_{i+1}$ for some $i\in[4]$ or is contained in $S_3$.
In particular, vertices of $S_5$ have no forward neighbors, so $u_1,u_2\notin S_5$.
Furthermore, for any two disjoint and non-adjacent edges $xy,wv$ such that $x,w\in S_i$ and $y,v\in S_{i+1}$, we have that if $x \prec w$, then $y \prec v$.
Therefore, the copy of $\rainbow$ cannot contain both edges that are between some sets $S_i$, $S_{i+1}$, and thus, at least one of the edges $u_1u_4$, $u_2u_3$ has to be contained in $S_3$.

If $u_1,u_4\in S_3$, then also $u_2,u_3\in S_3$ as we have $u_1 \prec u_2 \prec u_3 \prec u_4$.
However, all the edges inside $S_3$ share common vertex $s^3_k$, and thus $S_3$ cannot contain an induced copy of $\rainbow$.

So it must hold that $u_2,u_3\in S_3$ and $\{u_1,u_4\}\not\subseteq S_3$.
Therefore, either $u_1\in S_2$ and $u_4\in S_3$, or $u_1\in S_3$ and $u_4\in S_4$.
Since $u_2,u_3$ are adjacent, we have that one of $u_2,u_3$ has to be $s_3^k$, i.e., the universal vertex of $S_3$.
But then this vertex is adjacent to one of $u_1,u_4$ contained in $S_3$, so the copy of $\rainbow$ cannot be induced, a contradiction.
\end{claimproof}

It remains to verify that $P_{\langle \ell; j,k \rangle}$ satisfies the properties of a permutation gadget listed in \cref{def:permutation-gadget}. Property (P1) is obvious from the definition of $L$.

\begin{claim}\label{clm:rotationP2}
$P_{\langle \ell; j,k \rangle}$ satisfies property (P2).
\end{claim}
\begin{claimproof}
Let $f$ be a mapping $f: S_1\to \{1,2\}$.
We have to prove that $f$ can be extended to a coloring of $P_{\langle \ell; j,k \rangle}$, respecting lists $L$.
We will denote the extended coloring by $f$, too.

We will define $f$ vertex by vertex, according to the linear order of the vertex set.
At each step, for a vertex $s \in S_i$, we will choose for $s$ a color from $L(s)$ which is not used on any backward neighbor of $s$; note that all such neighbors are already colored when we color $s$.

Note that we only have to worry about edges between consecutive sets.
Indeed, the only edges in $S_3$ connect vertices with list $\{1,2\}$ with a vertex with list $\{3,4\}$, so they are irrelevant for the coloring.

\begin{description}
    \item[Vertices of $S_2$:] For $i\in [k-1]$, we color $s^2_i$ with color 2 (resp., 1) if $f(s^1_i)=1$ (resp.  $f(s^1_i)=2$).
        If $f(s^1_k)=1$, set the color of $s^2_k$ to 4 and the color of $s^2_{k+1}$ to 2.
        If $f(s^1_k)=2$, set the color of $s^2_k$ to 1 and the color of $s^2_{k+1}$ to 3.
        Finally, for $i \in [k+1,\ell]$, we set the color os $s^2_{i+1}$ to 2 (resp., 1) if $f(s^1_i)=1$ (resp., $f(s^1_i)=2$).

        Observe that the pair of colors appearing on $s^2_k,s^2_{k+1}$ is either $1,3$ or $2,4$.

        \item[Vertices of $S_3$:] For $i\in [k-1]$, we set the color of $s^3_i$ to 1 (resp. 2) if $f(s^2_i)=2$ (resp., $f(s^2_i)=2$).
        We also set the color of $s_3^k$ to 3 (resp., 4) if $f(s^2_k)=4$ and $f(s^2_{k+1})=2$ (resp., $f(s^2_k)=1$ and $f(s^2_{k+1})=3$).
        For $i \in [k+1,\ell]$, we set the color of $s^3_i$ to 1 (resp., 2) if $f(s_2^{i+1})=2$ (resp., $f(s_2^{i+1})=1$).        

        \item[Vertices of $S_4$:] For $i\in [j-1]$, we set the color of $s^4_i$ to 2 (resp., 1) if $f(s^3_i)=1$ (resp., $f(s^3_i)=2$).
        If $f(s_3^k)=3$, we set the color of $s^4_j$ to 4 and the color of $s^4_{j+1}$ to 2.
        If $f(s_3^k)=2$, we set the color of $s^4_j$ to 1 and the color of $s^4_{j+1}$ to 3.
        For $i \in [j, k-1]$, we set the color of $s^4_{i+2}$ to 2 (resp., 1) if $f(s^3_i)=1$ (resp., $f(s^3_i)=2$).
        For $i \in [k+1,\ell]$, we set the color of $s^4_{i+1}$ to 2 (resp., 1) if $f(s^3_{i})=1$ (resp. $f(s^3_{i})=2$).

        Similarly as in $S_2$, the pair of colors appearing on $s^4_j,s^4_{j+1}$ is either $4,2$ or $1,3$.

        \item[Vertices of $S_5$:] For every $i\in[j-1]$, we set the color of $s^5_i$ to 2 (resp., 1) if $f(s^4_i)=1$ (resp., $f(s^4_i)=2$).
        We also set the color of $s_5^j$ to 1 (resp., 2) if $f(s^2_j)=4$ and $f(s^2_{j+1})=2$ (resp., $f(s^2_j)=1$ and $f(s^2_{j+1})=3$).    
        For every $i\in[j+1,\ell]$, we set the color of $s^5_i$ to 2 (resp., 1) if $f(s^4_{i+1})=1$ (resp., $f(s^4_{i+1})=2$).
\end{description}
As, at the moment of coloring every vertex $v$, we chose a color distinct from the ones used for all backwards neighbors of $v$, we found a coloring of the gadget hat extends $f$.
This completes the proof of the claim.
\end{claimproof}

\begin{claim}\label{clm:rotationP3}
$P_{\langle \ell; j,k \rangle}$ satisfies property (P3).
\end{claim}
\begin{claimproof}
Let $f: V(S_1 \cup S_2 \cup S_3 \cup S_4 \cup S_5)\to [4]$ be a coloring of the gadget, respecting list $L$, and let $i\in [\ell]$.
We need to show that $f(s^1_i)=f(s^5_{\langle \ell; j,k \rangle(i)})$.
We consider cases, depending on the position of $i$ with respect to $j$ and $k$.

First assume that either (a) $i\in [j-1]$ or (b) $i \in [j, k-1]$, or (c) $i \in [k+1,\ell]$.
We have $\langle \ell; j,k \rangle(i) = i$ in cases (a) and (c), and $\langle \ell; j,k \rangle(i) = i+1$ in case (b).
In all cases, $P_{\langle \ell; j,k \rangle}$ has a 5-vertex path joining $s^1_i$ and $s^5_{\langle \ell; j,k \rangle(i)}$, whose all vertices have lists $\{1,2\}$. Depending on the case, the path is:
\begin{align*}
    \text{(a)} & \quad s^1_i - s^2_i - s^3_i - s^4_i - s^5_i \\
    \text{(b)} & \quad s^1_i - s^2_i - s^3_i - s^4_{i+2} - s^5_{i+1} \\
    \text{(c)} & \quad s^1_i - s^2_{i+1} - s^3_i - s^4_{i+1} - s^5_i.
\end{align*}
In any coloring of such a path, respecting the lists, the endpoints have the same color.

It remains to consider the case  $i=k$. Note that we have $\langle \ell; j, k \rangle (k) = j$.

Suppose first that $f(s^1_k)=1$ and consider the path $s^1_k - s^2_k - s^3_k - s^4_{j+1} - s^5_j$.
The lists of consecutive vertices are $\{1,2\}, \{1,4\}, \{3,4\}, \{2,3\}, \{1,2\}$.
Thus, if $f(s^1_k)=1$, the colors on this path are forced to be $1,4,3,2,1$.
In particular, $f(s^5_j) = 1$.

So now assume that $f(s^1_k)=2$ and consider the path $s^1_k - s^2_{k+1} - s^3_k - s^4_{j} - s^5_j$ with lists of consecutive vertices
$\{1,2\}, \{2,3\}, \{3,4\}, \{1,4\}, \{1,2\}$.
If $f(s^1_k)=2$, the colors on this path are forced to be $2,3,4,1,2$.
In particular, $f(s^5_j) = 2$.
\end{claimproof}

Thus, $P_{\langle \ell; j,k \rangle}$ is indeed a permutation gadget for ${\langle \ell; j,k \rangle}$, which was to be shown.
\end{proof}

Now, \cref{lem:permutation-gadget} follows directly from \cref{obs:compose-permutation-gadgets}, \cref{obs:compose-rotations}, and \cref{lem:rotation-gadget},

\subsubsection{Indicator gadgets}\label{sec:ind-gadget}
The next type of a gadget that we will need is an \emph{indicator}.

Throughout this subsection, let $n \in \N$ and let $I \subseteq \binom{[n]}{2}$ be a set of pairwise disjoint pairs  from of $[n]$, each of the form $\{i,i+1\}$, and let $N=n+|I|$.
Let us define a bijection $\gamma: [n] \cup I \to [n + |I|]$.
We will do it by specifying the relative order of images of particular elements of $[n] \cup I$.
The images of elements in $[n]$ are in the natural order, i.e., $\gamma(1) < \gamma(2) < \ldots < \gamma(n)$.
For each $\{i,i+1\} \in I$, the image of this pair is placed between $\gamma(i)$ and $\gamma(i+1)$.
Note that, as the pairs in $I$ are of the form $\{i,i+1\}$, $\gamma$ is well-defined.

\begin{definition}[Indicator gadget]\label{def:indicator}
    Let $c \in \{1,2\}$.
    An \emph{indicator gadget} (for $c,n$ and $I$) is a tuple $(\Ind_{c},(x_1,\ldots,x_n),$ $(y_1,\ldots,y_N),L)$,
    where $(\Ind_c,(x_1,\ldots,x_n),(y_1,\ldots,y_N))$ is a link and $L:V(\Ind_c)\to 2^{[4]}$ is a list function, such that:
    \begin{enumerate}[({I}1)]
        \item For every $i\in[n]$ and $j \in [N]$, it holds that $L(x_i)=L(y_j)=\{1,2\}$.
        
        \item Every $f:\{x_1,\ldots,x_n\}\to\{1,2\}$ can be extended to a coloring of the gadget so that the following holds.
        For every $\{i,i+1\}\in I$, if $\{f(x_i),f(x_{i+1})\}\neq \{c\}$, then $f(y_{\gamma(\{i,i+1\})}) \neq c$.
        
        \item For every coloring $f$ of the gadget, the following holds.
        For every $i\in[n]$, we have $f(x_i)=f(y_{\gamma(i)})$, and for every $\{i,i+1\}\in I$, if $f(x_i)=f(x_{i+1})=c$,
        then $f(y_{\gamma(\{i,i+1\})})=c$.
    \end{enumerate}    
\end{definition}
Note that each output vertex of the gadget corresponds either to an input vertex,
or to a pair in $I$; this correspondence is given by $\gamma$.
The gadget plays two roles. First, the coloring of input vertices is copied to their corresponding output vertices.
Second, for each pair in $I$, the color of its corresponding output vertex \emph{indicates}
if both input vertices representing this pair are colored $c$.

\begin{center}
\begin{figure}[t] 
    \begin{tikzpicture}[every node/.style={draw,circle,fill=white,inner sep=0pt,minimum size=8pt},every loop/.style={}]
\foreach \k in {1,2,3,4,5,6}
{
\node[fill=blue,opacity=0.5] (v1\k) at (-0.4+\k*0.4,0) {};
}

\foreach \k in {1,2,5,6,9,10}
{
\node[fill=blue,opacity=0.5] (v2\k) at (2.2+\k*0.4,0) {};
}
\foreach \k in {3,7}
{
\node[fill=orange,opacity=0.5] (v2\k) at (2.2+\k*0.4,0) {};
}
\foreach \k in {4,8}
{
\node[fill=yellow,opacity=0.5] (v2\k) at (2.2+\k*0.4,0) {};
}

\draw (v11) to [bend right] (v21);
\draw (v12) to [bend right] (v22);
\draw (v12) to [bend right] (v23);
\draw (v13) to [bend right] (v24);
\draw (v13) to [bend right] (v25);
\draw (v14) to [bend right] (v26);
\draw (v14) to [bend right] (v27);
\draw (v15) to [bend right] (v28);
\draw (v15) to [bend right] (v29);
\draw (v16) to [bend right] (v210);

\foreach \k in {1,2,4,5,7,8}
{
\node[fill=blue,opacity=0.5] (v3\k) at (6.4+\k*0.4,0) {};
}
\foreach \k in {3,6}
{
\node[fill=green,opacity=0.5] (v3\k) at (6.4+\k*0.4,0) {};
}

\draw (v21) to [bend left] (v31);
\draw (v22) to [bend left] (v32);
\draw (v23) to [bend left] (v33);
\draw (v24) to [bend left] (v33);
\draw (v25) to [bend left] (v34);
\draw (v26) to [bend left] (v35);
\draw (v27) to [bend left] (v36);
\draw (v28) to [bend left] (v36);
\draw (v29) to [bend left] (v37);
\draw (v210) to [bend left] (v38);

\foreach \k in {1,2,3,4,5,6,7,8}
{
\node[fill=blue,opacity=0.5] (v4\k) at (9.8+\k*0.4,0) {};
}

\foreach \k in {1,2,3,4,5,6,7,8}
{
\draw (v3\k) to [bend right] (v4\k);
}

\foreach \k in {1,2,3,4,5,6,7,8}
{
\node[fill=blue,opacity=0.5] (v5\k) at (13.2+\k*0.4,0) {};
\draw (v4\k) to [bend left] (v5\k);
}

\node[fill=blue,opacity=0.5,label=right:\footnotesize{$\{1,2\}$}] (l1) at (0.5,2) {};
\node[fill=orange,opacity=0.5,label=right:\footnotesize{$\{1,3\}$}] (l2) at (0.5,1.5) {};
\node[fill=yellow,opacity=0.5,label=right:\footnotesize{$\{1,4\}$}] (l3) at (2.5,2) {};
\node[fill=green,opacity=0.5,label=right:\footnotesize{$\{1,3,4\}$}] (l4) at (2.5,1.5) {};

\end{tikzpicture}
     \caption{Gadget $\Ind_{1}$ for $n=6$ and $I=\{\{2,3\},\{4,5\}\}$.}
    \label{fig:ind-gadget}
\end{figure}
\end{center}

In the following lemma we show how to construct indicator gadgets; consult also~\cref{fig:ind-gadget}.

\begin{lemma}\label{lem:ind-gadget}
     Let $c \in \{1,2\}$, let $n \in\N$ and let $I\subseteq \binom{[n]}{2}$ be a set of pairwise disjoint pairs from of $[n]$, each of the form $\{i,i+1\}$.
     In time polynomial in $n$, we can construct an indicator gadget for $c,n$, and $I$.
\end{lemma}
\begin{proof}
    By symmetry, assume that $c=1$.    
    Let $\gamma$ be the function defined for $n$ and $I$ as before.
    
    The vertex set of the gadget is partitioned into five sets $V_1,\ldots,V_5$, where, for all $i \in [4]$, all vertices of $V_{i}$ precede all vertices from $V_{i+1}$.
    We will construct the sets one by one.
    
\begin{description}
    \item[Set $V_1$:] We start with introducing the vertices $v^1_1,\ldots,v^1_n$ (in that order), all with list $\{1,2\}$.
    \item[Set $V_2$:] Next, we introduce vertices $v^2_1,\ldots,v^2_n$ (in that order), all with list $\{1,2\}$.
                Next, for every $\{i,i+1\}\in I$, we add vertices $u^2_i,w^2_i$, with lists, respectively, $\{1,3\}$ and $\{1,4\}$.
                We position them in the order so that $v^2_i \prec u^2_i \prec w^2_i \prec v^2_{i+1}$.
                For every $i\in [n]$, we add the edge $v^1_iv^2_i$, and for every $\{i,i+1\}\in I$, we add edges $v^1_iu^2_i$ and $v^1_{i+1}w^2_i$.

    \item[Set $V_3$:] We introduce vertices $v^3_1,\ldots,v^3_n$ (in that order), all with list $\{1,2\}$.
                For every $\{i,i+1\}\in I$, we add a vertex $z^3_i$, with list $\{1,3,4\}$.
                We insert $z^3_i$ between $v^3_i$ and $v^3_{i+1}$.
                For every $i\in [n]$, we add the edge $v^2_iv^3_i$, and for every $\{i,i+1\}\in I$, we add edges $u^2_iz^3_i$ and $w^2_iz^3_i$.

    \item[Set $V_4$:] We introduce  vertices $v^4_1,\ldots,v^4_n$ (in that order), all with list $\{1,2\}$.
                For every $\{i,i+1\}\in I$, we add vertex $z^4_i$ with list $\{1,2\}$.
                We insert $z^4_i$ between $v^4_i$ and $v^4_{i+1}$.
                For every $i\in [n]$, we add the edge $v^3_iv^4_i$, and for every $\{i,i+1\}\in I$, we add the edge $z^3_iz^4_i$.

    \item[Set $V_5$:] We introduce vertices $v^5_1,\ldots,v^5_n$ (in that order), all with list $\{1,2\}$.
                For every $\{i,i+1\}\in I$, we add a vertex $z^5_i$ with list $\{1,2\}$.
                We insert $z^5_i$ between $v^5_i$ and $v^5_{i+1}$.
                For every $i\in [n]$, we add the edge $v^4_iv^5_i$, and for every $\{i,i+1\}\in I$, we add the edge $z^4_iz^5_i$.
\end{description}
This completes the construction of $\Ind_1$; the input vertices are $V_1$ and the output vertices are $V_5$.
Note that, for each $i \in [n]$, the $\gamma(i)$-th output vertex is $v^5_i$,
and for each $\{i,i+1\} \in I$, the $\gamma(\{i,i+1\})$-th output vertex is $z^5_i$.

Let us argue that $\Ind_1$ is an indicator gadget for 1, $n$, $I$.

\begin{claim}
    $\Ind_{1}$ is $\rainbow$-free.
\end{claim}
\begin{claimproof}
    Observe that each of the sets $V_1,\ldots,V_5$ is independent and the only edges are between pairs of consecutive sets.
    Moreover, by the construction, for two disjoint edges $uv,xy$ such that $u,x\in V_{i}$ and $v,y\in V_{i+1}$, it holds that if $u \prec x$, then $v \prec y$, so $\Ind_{1}$ does not contain $\rainbow$ as a subgraph (even not necessarily induced).
\end{claimproof}

Now, let us verify properties from \cref{def:indicator}. Property (I1) is obvious.

\begin{claim}\label{clm:I2}
    $\Ind_{1}$ satisfies (I2) for $\gamma$.
\end{claim}
\begin{claimproof}
Let $f:\{v^1_1,\ldots,v^1_n\}\to \{1,2\}$.

We aim to extend $f$ to a coloring of the whole gadget, so that, for every $\{i,i+1\} \in I$,
if $\{ f(v^1_i),f(v^1_{i+1}) \} \neq \{1\}$, then the color of $v^5_{\gamma(\{i,i+1\})}$ is 2.
Slightly abusing the notation, we will call the extended coloring $f$.

For $i\in [n]$, if $f(v^1_i)=1$ (resp., $f(v^1_i)=2$), then we set $f(v^2_i)=f(v^4_i)=2$ (resp., $f(v^2_i)=f(v^4_i)=1$) and $f(v^3_i)=f(v^5_i)=1$ (resp., $f(v^3_i)=f(v^5_i)=1$).
Note that the vertices $v^p_i$ for $p \in \{2,3,4,5\}$ have no other neighbors of the gadget.

Now consider a pair $\{i,i+1\}\in I$.
If $\{f(v^1_i),f(v^1_{i+1})\}=\{1\}$, then we set $f(u^2_i)=3$ and $f(w^2_i)=4$, and we set $f(z^3_i)=f(z^5_i)=1$ and $f(z^4_i)=2$.
If $\{f(x'_i),f(x'_{i+1})\}\neq\{1\}$, then at least one of $v^1_i,v^1_{i+1}$ is colored with $2$, and thus at least one of $u^2_i,w^2_i$ can be colored with $1$. Then $z^3_i$ can be colored with a color from $\{3,4\}$.
We finish the coloring by setting $f(z^4_i)=1$ and $f(z^5_i)= 2$, as was required for this $\{i,i+1\}$. 
Thus, the property (I2) is indeed satisfied. 
\end{claimproof}

\begin{claim}\label{clm:I3}
    $\Ind_{1}$ satisfies (I3) for $\gamma$.
\end{claim}
\begin{claimproof}
Let $f$ be a coloring of $(\Ind_{1},L)$.
Note that, for each $i \in [n]$, there is a path $v^1_i - v^2_i - v^3_i - v^4_i - v^5_i$ with all vertices with list $\{1,2\}$.
Consequently, the color of $v^5_i$ is the same as the color of $v^1_i$, and these two vertices are paired by $\gamma$.

Now let $\{i,i+1\}\in I$ be such that $f(v^1_i)=f(v^1_{i+1})=1$.
Then it must hold that $f(u^2_i)=3$ and $f(w^2_i)=4$ as $u^2_i,w^2_i$ are adjacent, respectively, to $v^1_i$ and $v^1_{i+1}$,
and have lists, respectively, $\{1,3\}$ and $\{1,4\}$.
Since $z^3_i$ is adjacent to both $u^2_i,w^2_i$ and has list $\{1,3,4\}$, it holds $f(z^3_i)=1$.
Thus, the path $z^3_i - z^4_i - z^5_i$ whose all vertices have list $\{1,2\}$ forces the color of $z^5_i$ to 1.
This vertex precisely corresponds to $\{i,i+1\}$ via $\gamma$.
So, property (I3) is satisfied.
\end{claimproof}

Thus, $\Ind_1$ is indeed an indicator gadget.
This completes the proof of the lemma.
\end{proof}

\subsubsection{NOT-$cc$ gadgets}\label{sec:notcc-gadget}
The last type of a gadget is a NOT-$cc$ gadget, for $c \in \{1,2\}$.

Throughout this subsection, let $N \in \N$ and let $I \subseteq \binom{[N]}{2}$ be a set of pairwise disjoint pairs from of $[N]$, each of the form $\{i,i+1\}$.
Let $n = N - |I|$; note that $n > 0$.
Let $\delta : [n] \to [N]$ be an injection defined as follows.
The image of this function is the set of these $i$, for which $\{i,i+1\} \notin I$.
Furthermore, we have $\delta(1) < \delta(2) < \ldots < \delta(n)$.
Note that $\delta$ is well-defined.

\begin{definition}[NOT-$cc$]\label{def:ncc}
    Let $c \in \{1,2\}$.
    An \emph{NOT-$cc$ gadget} (for $N$ and $I$) is a tuple $(B_{c},(x_1,\ldots,x_N),(y_1,\ldots,y_n),L)$,
    where $(B_c,(x_1,\ldots,x_N),(y_1,\ldots,y_n))$ is a link and $L:V(B_c)\to 2^{[4]}$ is a list function, such that:
    \begin{enumerate}[({B}1)]
        \item For every $i\in[N]$ and $j \in [n]$, it holds that $L(x_i)=L(y_j)=\{1,2\}$.
        
        \item Every $f:\{x_1,\ldots,x_N\}\to\{1,2\}$ such that for every $\{i,i+1\} \in I$ it holds that $\{f(x_1),f(x_2)\} \neq \{c\}$,
        can be extended to a coloring of the gadget.
        
        \item For every coloring $f$ of the gadget, the following holds.
        For every $i\in[n]$, we have $f(y_i)=f(x_{\delta(i)})$, and for every $\{i,i+1\}\in I$ we have $\{ f(x_i),f(x_{i+1}) \} \neq \{c\}$.
    \end{enumerate}    
\end{definition}
Note that there are two types of input vertices: those $x_i$ for which $\{i,i+1\} \notin I$, and the remaining ones.
The former ones have corresponding output vertices; this correspondence is given by $\delta$.
The remaining input vertices do not have corresponding output vertices.

Again, the gadget plays two roles. First, it transfers the coloring of  input vertices of the first type to their corresponding output vertices.
Second, for each pair in $I$, the gadget ensures that input vertices corresponding to that pair are not both colored $c$.

\begin{center}
    \begin{figure}[t]
        \centering
   
\begin{tikzpicture}[every node/.style={draw,circle,fill=white,inner sep=0pt,minimum size=8pt},every loop/.style={}]


\foreach \k in {1,2,3,4,5,6,7,8}
{
\node[fill=blue,opacity=0.5] (v6\k) at (-0.4+\k*0.4,0) {};
}

\foreach \k in {1,2,5,6,9,10}
{
\node[fill=blue,opacity=0.5] (v7\k) at (3+\k*0.4,0) {};
}
\foreach \k in {3,7}
{
\node[fill=orange,opacity=0.5] (v7\k) at (3+\k*0.4,0) {};
}
\foreach \k in {4,8}
{
\node[fill=yellow,opacity=0.5] (v7\k) at (3+\k*0.4,0) {};
}

\draw (v61) to [bend left] (v71);
\draw (v62) to [bend left] (v72);
\draw (v63) to [bend left] (v73);
\draw (v64) to [bend left] (v74);
\draw (v64) to [bend left] (v75);
\draw (v65) to [bend left] (v76);
\draw (v66) to [bend left] (v77);
\draw (v67) to [bend left] (v78);
\draw (v67) to [bend left] (v79);
\draw (v68) to [bend left] (v710);

\foreach \k in {1,2,4,5,7,8}
{
\node[fill=blue,opacity=0.5] (v8\k) at (7.2+\k*0.4,0) {};
}
\foreach \k in {3,6}
{
\node[fill=green,opacity=0.5] (v8\k) at (7.2+\k*0.4,0) {};
}

\draw (v71) to [bend right] (v81);
\draw (v72) to [bend right] (v82);
\draw (v73) to [bend right] (v83);
\draw (v74) to [bend right] (v83);
\draw (v75) to [bend right] (v84);
\draw (v76) to [bend right] (v85);
\draw (v77) to [bend right] (v86);
\draw (v78) to [bend right] (v86);
\draw (v79) to [bend right] (v87);
\draw (v710) to [bend right] (v88);

\foreach \k in {1,2,3,4,5,6}
{
\node[fill=blue,opacity=0.5] (v9\k) at (10.6+\k*0.4,0) {};
\node[fill=blue,opacity=0.5] (v10\k) at (13.4+\k*0.4,0) {};
\draw (v9\k) to [bend right] (v10\k);
}

\draw (v81) to [bend left] (v91);
\draw (v82) to [bend left] (v92);
\draw (v84) to [bend left] (v93);
\draw (v85) to [bend left] (v94);
\draw (v87) to [bend left] (v95);
\draw (v88) to [bend left] (v96);

\node[fill=blue,opacity=0.5,label=right:\footnotesize{$\{1,2\}$}] (l1) at (0.5,2) {};
\node[fill=orange,opacity=0.5,label=right:\footnotesize{$\{1,3\}$}] (l2) at (0.5,1.5) {};
\node[fill=yellow,opacity=0.5,label=right:\footnotesize{$\{1,4\}$}] (l3) at (2.5,2) {};
\node[fill=green,opacity=0.5,label=right:\footnotesize{$\{3,4\}$}] (l4) at (2.5,1.5) {};

\end{tikzpicture}
     \caption{A NOT-11 gadget for $n=6$ and $I=\{\{2,3\},\{4,5\}\}$.}
    \label{fig:notcc-gadget}
\end{figure}
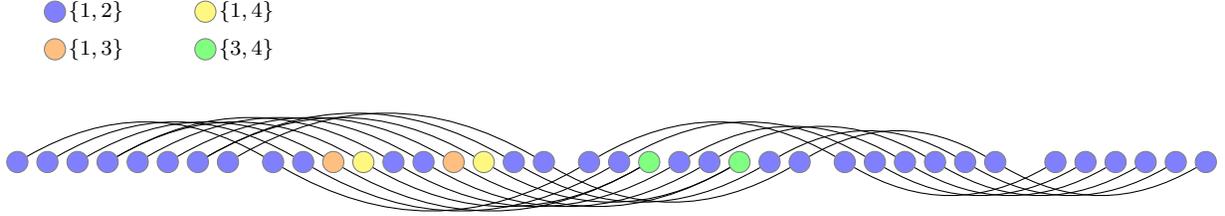
\end{center}

In the following lemma we show how to construct NOT-$cc$ gadgets; consult also~\cref{fig:notcc-gadget}.

\begin{lemma}
     Let $c \in \{1,2\}$, let $N \in\N$ and let $I\subseteq \binom{[N]}{2}$ be a set of pairwise disjoint pairs from of $[N]$, each of the form $\{i,i+1\}$. Let $n = N - |I|$.
     In time polynomial in $N$, we can construct a NOT-$cc$ gadget for $I$.
\end{lemma}
\begin{proof}
By symmetry, assume that $c=1$.
The gadget is actually an induced subgraph of the indicator gadget, so let us just describe the differences.

For $N$ and $I$, consider the gadget $\Ind_1$ given by \cref{lem:ind-gadget}.
We use names of vertices from the proof of that lemma.

First, for each $\{i,i+1\} \in I$, we remove 1 from the list of the vertex $z^3_3$, so that it is $\{3,4\}$ instead of $\{1,3,4\}$.
Next, we do not introduce vertices $z^4_i$ and $z^5_i$.
Finally, for each pair $\{i,i+1\} \in I$, we remove vertices $v^2_i,v^3_i,v^4_i,v^5_i$.
Note that after this process, the number of vertices remaining in $V_5$ is $N - |I| = n$.

We call the obtained graph $B_1$.
The input and output of the gadget are, respectively, $V_1$ and $V_5 \cap V(B_1)$.
Note that the vertices removed from $V_5$ are precisely the ones that are not in the image of $\delta$. 

Clearly, $B_1$ is $\rainbow$-free, as it is an induced subgraph of $\Ind_1$.
As we removed 1 from the list of $z_3^i$, for all $\{i,i+1\} \in I$, properties (B2) and (B3) follow directly from, respectively, the proofs of \cref{clm:I2} and \cref{clm:I3}.

This completes the proof of the lemma.
\end{proof}

\subsubsection{Constructing NOT-$ccc$ gadgets}

Finally, we can prove the following statement, which would finish the proof of \cref{lem:clause-gadget} and thus, of \cref{thm:m7}.

\begin{lemma}\label{lem:nai-gadget}
Let $c \in \{1,2\}$, let $n\in\N$ and let $I\subseteq \binom{[n]}{3}$ be a set of pairwise disjoint subsets of $[n]$, each of size $3$.
In time polynomial in $n$, we can construct a NOT-$ccc$ gadget for $I$.
\end{lemma}

Let us first present the general idea of our approach.
Let $x_1,\ldots,x_n$ be the input vertices of the gadget.
The gadget is supposed to forbid some triples of vertices to be colored only with color 1.
However, doing this directly seems difficult, while keeping $\rainbow$-freeness.
Thus, we do it in two steps.

For each triple $\{i,j,k\} \in I$, we create a new vertex $x_{i,j}$, so that if both vertices $x_i$ and $x_j$ are colored $c$,
then the color of $x_{i,j}$ is forced to be $c$, and otherwise $x_{i,j}$ can be colored with the other color in $\{1,2\}$.
This is precisely the behavior that can be forced by an indicator gadget.
Then we make sure that $x_{i,j}$ and $x_k$ (actually, a vertex whose color is equal to the color of $x_k$; here we use the second role of an indicator) are not both colored $c$. This can be done using the NOT-$cc$ gadget.

There is one last thing to do: recall that the indicator and the NOT-$cc$ gadgets can only operate on pairs of consecutive vertices,
and $\{i,j,k\}$ do not have to be consecutive.
However, this can be easily obtained by reshuffling the vertices using appropriate permutation gadgets.

Summing up, the NOT-$cc$ is obtained by chaining a permutation gadget, an indicator gadget, another permutation gadget, a NOT-$cc$ gadget, and, finally, yet another permutation gadget, to restore the original ordering of vertices.

Now, let us describe the construction be more formally.

\begin{proof}[Proof of \cref{lem:nai-gadget}]
By symmetry, assume that $c = 1$.
Let $N:=n+|I|$.

Let $\sigma_1$ be a permutation of $[n]$, so that, for every $\{i,j,k\} \in I$ where $i < j < k$ we have $\sigma_1(j) = \sigma_1(i)+1$. Note that there are many such permutations, but any of them works for us.
We call \cref{lem:permutation-gadget} to obtain a permutation gadget $P_{\sigma_1}$ for $\sigma_1$.

Next, let $I'$ be the set of pairs $\{\sigma_1(i), \sigma_1(j)\}$, where $\{i,j,k\} \in I$ and $i < j < k$.
Note that $|I'| = |I|$ and pairs of $I'$ are pairwise disjoint.
Let $\gamma$ be defined for $n$ and $I'$ as in \cref{sec:ind-gadget}.
We call \cref{lem:ind-gadget} for $1, n$, and $I'$ to obtain an indicator gadget $\Ind_1$.

Next, let $\sigma_2$ be a permutation of $[N]$ such that for every $\{i,j,k\} \in I$ with $i < j < k$, we have $\sigma_2(\gamma(\sigma_1(i),\sigma_1(j))) = \sigma_2(\gamma(\sigma_1(k)))+1$.
We call \cref{lem:permutation-gadget} to obtain a permutation gadget $P_{\sigma_2}$ for $\sigma_2$.

Let $I''\subseteq \binom{[N]}{2}$ be the set of pairs $\{\sigma_2(\gamma(\sigma_1(i),\sigma_1(j))), \sigma_2(\gamma(\sigma_1(k)))\}$ for $\{i,j,k\}\in I$ where $i < j < k$.
Note that these pairs are pairwise disjoint, and each contains consecutive elements.
Let $\delta$ be the function defined in \cref{sec:notcc-gadget}.
We call \cref{lem:nai-gadget} for $N$ and $I''$ to obtain a NOT-11 gadgets $B_1$.

Finally, let $\sigma_3: [n]\to [n]$ be the permutation which is the inverse of
$\delta^{-1} \circ \sigma_2 \circ \gamma \circ \sigma_1$ -- as each of these functions is injective, $\sigma_3$ is well-defined.
We introduce permutation gadget $P_{\sigma_3}$ given by \cref{lem:permutation-gadget} for $\sigma_3$.

Now, the gadget $C^1$ is obtained by chaining all gadgets constructed so far, in a given order: $P_{\sigma_1}, \Ind_1, P_{\sigma_2}, B_1, P_{\sigma_2}$.
The input vertices of $C^1$ are the input vertices of $P_{\sigma_1}$, denote them by $x_1,\ldots,x_n$.
The output vertices of $C^1$ are the output vertices of $P_{\sigma_3}$, denote them by $y_1,\ldots,y_n$.
We also denote by
\begin{align*}
     & u_1,\ldots,u_n &&  \text{ the output vertices of $P_{\sigma_1}$ (equivalently, the input vertices of $\Ind_1$)},\\
     & v_1,\ldots,v_N &&  \text{ the output vertices of $\Ind_1$ (equivalently, the input vertices of $P_{\sigma_2}$)},\\
     & w_1,\ldots,w_N &&  \text{ the output vertices of $P_{\sigma_2}$ (equivalently, the input vertices of $B_1$)},\\
     & z_1,\ldots,z_n &&  \text{ the output vertices of $B_1$ (equivalently, the input vertices of $P_{\sigma_3}$).}
\end{align*}

Note that since $\sigma_3 \circ \delta^{-1} \circ \sigma_2 \circ \gamma \circ \sigma_1$ is the identity function, the input vertices correspond to output vertices in a natural order.

Let us argue that $C^1$ is indeed a NOT-111 gadget.
Repeated application of \cref{obs:m7-combining} yields that it is $\rainbow$-free.
So, we are left with checking the properties listed in \cref{def:notccc}.
Property (C1) is obvious. 

\begin{claim}
    $C^1$ satisfies property (C2).    
\end{claim}
\begin{claimproof}
Now let $f:\{x_1,\ldots,x_n\}\to \{1,2\}$ be such that for every $\{i,j,k\}\in I$, it holds that $\{f(x_i),f(x_j),f(x_k)\}\neq \{1\}$.
We extend $f$ to the remaining vertices as follows; we will denote the extended coloring by $f$.

First, using property (P2) for $P_{\sigma_1}$, we extend  $f$ to the vertices of $P_{\sigma_1}$. Note that, by (P3), for each $i \in [n]$, we have $f(x_i) = f(u_{\sigma_1(i)})$.
In particular, for each triple $\{i,j,k\}$, vertices $u_{\sigma(i)},u_{\sigma(j)},u_{\sigma(k)}$ are not all colored 1.

Now, by property (I2), the coloring can be further extended to all vertices of $\Ind_1$.
By (I2) and (I3) we can ensure that, for each $\{i,j,k\} \in I$, where $i < j < k$, if $\{f(x_i), f(x_j)\}=\{1\}$, then $f(v_{\gamma(\sigma_1(i),\sigma_1(j))})=1$, and if 
if $\{f(x_i),f(x_j)\} \neq \{1\}$, then $f(v_{\gamma(\sigma_1(i),\sigma_1(j))})=2$.
Note that in the former case, we have $f(v_{\gamma(\sigma_1(k))})=2$, as the vertices $x_i,x_j,x_k$ are not all colored 1.

Next, we extend the coloring to $P_{\sigma_2}$, again using property (P2).
Note that by (P3), the colors on input vertices of $P_{\sigma_2}$ are copied to their corresponding output vertices.
Recall that, due to the definition of $\sigma_2$, for all $\{i,j,k\} \in I$ where $i < j < k$, vertices $w_{\sigma_2(\gamma(\sigma_1(i),\sigma_1(j)))}$ and $w_{\sigma_2(\gamma(\sigma_1(k)))}$ are consecutive.

Now, by property (B1), we can extend the coloring to the vertices of $B_1$.
Here we use the fact that for all $\{i,j,k\} \in I$, where $i < j < k$,
the vertices $w_{\sigma_2(\gamma(\sigma_1(i),\sigma_1(j)))}$ and $w_{\sigma_2(\gamma(\sigma_1(k)))}$ are not both colored one.

Finally, by property (P2), we extend the coloring to the vertices of $P_{\sigma_3}$,
and thus to all vertices of $C^1$.
\end{claimproof}

\begin{claim}
    $C^1$ satisfies property (C3).    
\end{claim}
\begin{claimproof}
Let $f$ be a coloring of $V(C^1)$ respectng lists $L$.

Property (C3) (a) follows directly by combining properties (P3), (I3), (P3), (B3), and (P3) for, respectively $P_{\sigma_1}, \Ind_1, P_{\sigma_2}, B_1, P_{\sigma_3}$.

For property (C3) (b), suppose that for some $\{i,j,k\} \in I$ we have $f(x_i)=f(x_j)=f(x_k)=1$.
As in the previous paragraph, we observe that $f(u_{\sigma_1(i)})=f(u_{\sigma_1(j)})=f(u_{\sigma_1(k)})=1$.

By the property (I3) of $\Ind_1$ we observe that $f(v_{\gamma(\{ \sigma_1(i),\sigma_1(j) \})})=f(v_{\gamma(\sigma_1(k))})=1$.

Now, by property (P3) of $P_{\sigma_2}$, we obtain that $f(w_{\sigma_2(\gamma(\{ \sigma_1(i),\sigma_1(j) \}))})=f(v_{\sigma_2(\gamma(\sigma_1(k)))})$.
However, this is a contradiction with property (B3) for $B_1$.
Thus, property (C3) holds for $C^1$.
\end{claimproof}

This completes the proof.
\end{proof}

Now, \cref{lem:clause-gadget} follows directly by combining \cref{lem:nai-gadget} and \cref{obs:from-nai-to-nae}.
This completes the proof of \cref{thm:m7}.

\section{Unbounded number of colors}
In this section we summarize the situation for \listcol, i.e., the variant with unbounded number of colors.
Unsurprisingly, all non-trivial cases turn out to be \textsf{NP}-hard.

\begin{theorem}\label{thm:unbounded}
    Let $H$ be an ordered graph on at least two vertices.
    The \listcol problem in $H$-free ordered graphs is polynomial-time solvable if $H$ has two vertices, and \textsf{NP}-hard otherwise.
\end{theorem}
\begin{proof}
    If $H$ consists of a single edge, then $H$-free graphs are edgeless and thus the problem is polynomial-time solvable: the only no-instances are those that contain a vertex with an empty list.

    If $H$ has two vertices and no edges, then $H$-free graphs are cliques.
    Solving \listcol on such instances is equivalent to finding a saturating matching in a bipartite graph with sides representing, respectively, vertices of the instance and the colors.

    So let us discuss hard cases. Assume that $H$ has three vertices.

    Golovach, Paulusma, and Song~\cite{DBLP:journals/iandc/GolovachPS14} showed that \listcol is \textsf{NP}-hard on graphs obtained from a complete graph by removing a matching.
    Any vertex ordering of such a graph excludes every three-vertex graph $H$ with at most one edge.

    Jansen and Scheffler~\cite{DBLP:journals/dam/JansenS97} proved that \listcol is \textsf{NP}-hard in complete split graphs. Ordering the vertices of such an instance so that the independent set precedes the clique excludes $H = \jj$.

    For the last two cases, i.e., $H  \in \{ \pthree, \ttriangle\}$, already \lcol{3} is \textsf{NP}-hard in $H$-free ordered graphs~\cite{DBLP:journals/siamdm/HajebiLS24}.
\end{proof}

\section{Conclusion}
Let us conclude the paper with pointing out some possible directions for future research.
An obvious one is to fill the gaps in~\cref{tab:summary}. We believe obtaining a full dichotomy here is much easier than the analogous question for unordered graphs.
In particular, we believe that understanding the complexity of \lcol{$k$} for $k \geq 3$ in $\wave$-free ordered graphs is a specific and intriguing problem.

Second, we believe it is interesting to investigate the complexity of the (non-list) \col{$k$} problem in hereditary classes of ordered graphs.
While algorithms for \lcol{$k$} clearly carry over to \col{$k$} this is not the case for hardness reductions.
Note that our proofs crucially use lists, and a standard way of simulating lists by adding a $k$-clique and using the edges to this clique to forbid certain colors introduces new induced subgraphs.

Finally, we remark that the hardness reduction from \cref{thm:jp}, as well as all hardness reductions  proof for \lcol{3} in~\cite{DBLP:journals/siamdm/HajebiLS24} produce instances of linear size. Consequently, they exclude the existence of subexponential-time algorithms for corresponding problems.
On the other hand, the reduction in \cref{thm:m7} is \emph{quadratic}, and thus, it only excludes algorithms with running time $2^{o(\sqrt{n})}$, where $n$ is the number of vertices of the input.
The bottleneck is the construction of a permutation gadget, which has $\Theta(n^2)$ vertices.
It would be interesting to improve this construction to a linear one, or provide a subexponential-time algorithm for the problem.

We believe more in the former outcome. Indeed, recall that our construction mimics the selection sort algorithm, which is quadratic. Perhaps an inspiration for a better gadget could come from faster sorting algorithms?

\bibliographystyle{plain}
\bibliography{main}

\end{document}